\documentclass[smallextended,numbook,runningheads]{svjour3}
\usepackage{amsmath}

\smartqed  
\usepackage{amsfonts,amssymb}
\usepackage{epsfig}
\usepackage{booktabs}
\usepackage[numbers]{natbib}
\usepackage{xcolor}
\usepackage{epstopdf}
\usepackage{hyperref}
\hypersetup{colorlinks=true, urlcolor=blue, citecolor=blue, linkcolor=blue}

\usepackage[capitalise]{cleveref}

\journalname{}
\date{ \phantom{b} \vspace{45mm}\phantom{e}}
\def\makeheadbox{\hspace{-4mm}Version of 20 January 2025}

\def\R{{\mathbb R}}

\def\e{{\mathrm e}}

\def\iu{\mathrm{i}}
\def\eps{\varepsilon}
\def\calA{{\cal O}\!\!\!\:\iota}
\def\calH{{\cal H}}

\def\calM{{\cal M}}
\def\calQ{{\cal Q}}
\def\calT{{\cal T}}

\def\wt{\widetilde}

\def\Re{{\mathrm{Re}\,}}

\begin{document}

\title{Regularized dynamical parametric approximation of\\ stiff evolution problems}

\titlerunning{Regularized dynamical parametric approximation for stiff problems}

\author{Christian Lubich$^1$, Jörg Nick$^2$}
\authorrunning{C.\ Lubich, J.\ Nick}

\institute{
$^1$~Mathematisches Institut, Auf der Morgenstelle 10, Univ.\ T\"ubingen, D-72076 T\"ubingen, Germany.
\email{Lubich@na.uni-tuebingen.de}\\
$^2$~Seminar für Angewandte Mathematik, Rämistrasse 101, ETH Z\"urich, CH-8092 Z\"urich, Switzerland.
\email{joerg.nick@math.ethz.ch}\\
}

\date{ }

\maketitle

 \begin{abstract} 
 Evolutionary deep neural networks have emerged as a rapidly growing field of research. This paper studies numerical integrators for such and other classes of nonlinear parametrizations \( u(t) = \Phi(\theta(t)) \) where the evolving parameters \( \theta(t) \) are to be computed. The primary focus is on tackling the challenges posed by the combination of stiff evolution problems and irregular parametrizations, which typically arise with neural networks, tensor networks, flocks of evolving Gaussians, and in further cases of overparametrization.
 We propose and analyse regularized parametric versions of the implicit Euler method and higher-order implicit Runge--Kutta methods for the time integration of the parameters in nonlinear approximations to evolutionary partial differential equations and large systems of stiff ordinary differential equations. At each time step, an ill-conditioned nonlinear optimization problem is solved approximately with a few regularized Gauß--Newton iterations. Error bounds for the resulting parametric integrator are derived by relating the computationally accessible Gauß--Newton iteration for the parameters to the computationally inaccessible Newton iteration for the underlying non-parametric time integration scheme.
 The theoretical findings are supported by numerical experiments that are designed to show key properties of the proposed parametric integrators.

\end{abstract}

\section{Introduction}

Deep neural networks, tensor networks, and sums of complex Gaussians are nonlinear parametrizations of functions, $u=\Phi(\theta)$ with parameters $\theta$, where  typically the parametrization map $\Phi$ is irregular: the derivative $\Phi'(\theta)$ has arbitrarily small singular values and possibly varying rank. This is a characteristic of overparametrization, which is commonly encountered in applications. 

In this paper, we study the time integration of evolution equations $\dot y =f(y)$ using such irregular nonlinear parametrizations. Here, the solution $y(t)$ is approximated by a nonlinear parametrization $u(t)=\Phi(\theta(t))$ with time-dependent parameters. 
A widely used approach, known as the Dirac--Frenkel time-dependent variational principle in physics and chemistry, 
is to minimize the defect $\| \dot u - f(u) \|=\|\Phi'(\theta)\dot \theta - f(\Phi(\theta))\|$. This determines the time derivative of the parameters as a function of the parameters 
via a linear least-squares problem, which is, however, ill-posed in the irregular situation considered here.  

Our approach can be summarized as follows: We regularize the linear least-squares problem for the time derivative of the parameters, as in \cite{FLLN24}. Then, discretizing the appearing time derivatives yields a regularized nonlinear least-squares problem for the parametric approximation $u_n=\Phi(\theta_n)$ in the $n$th time step.  This is then solved approximately by a few Gau\ss--Newton iterations. 

The parametric time integration methods studied here are intended to be used for evolutionary partial differential equations and for prohibitively large systems of stiff ordinary differential equations. 
Such problems require {\it implicit} time discretization. We consider parametric versions of the implicit Euler method, the implicit midpoint rule, and higher-order implicit Runge--Kutta methods, in particular Radau and Gau\ss\ methods. 

Extending the terminology of \cite{ZCVP24}, which distinguishes between Discretize-then-Optimize (DtO) and Optimize-then-Discretize (OtD) schemes, the approach taken here is RtDtO, with R for regularization. (It can also be reinterpreted as DtOtR.) We emphasize and analyse the important role of regularization, following up on \cite{FLLN24}, where an RtOtD approach based on explicit time-stepping methods was investigated for non-stiff problems.

We analyse the parametric time integration methods in a Hilbert space setting of semilinear evolution equations $\dot y = Ay + g(y)$, where $A$ generates a contraction semigroup and $g$ has a moderate Lipschitz constant. Note that grid-based space discretization of partial differential equations is circumvented by the mesh-free nonlinear parametrization. 

In the error analysis, we compare the Gau\ss--Newton iterates of the  parametric implicit integrator with the Newton iterates of the integrator applied directly to the evolution equation. The former iterates are actually computed and the latter serve as a theoretical construct and are not computed, typically cannot even be computed.  However, their errors are well understood. A main result of our error analysis shows that under a restriction between the stepsize and the regularization parameter, the parametric Gau\ss--Newton iterates and the non-parametric Newton iterates differ only by a term bounded by the defect sizes (or residual norms or losses in other terminologies) in the regularized linear least squares problems appearing in the Gau\ss--Newton iterations.

This result allows us to derive a global error bound that is independent of the stiffness and is fully determined by the error bound of the abstract, non-parametric integrator (this error bound is of a known approximation order in the step size) and the maximal defect size occurring in the iterations (available {\it a posteriori}). We note that smallness of the defect size can be viewed as a modeling assumption. If this assumption is not met, the model (say, out of the three classes mentioned in the beginning) can usually be refined with more parameters, making the defect smaller and the parametrization map even more irregular.

Numerical experiments with parametrizations by neural networks illustrate properties of the proposed regularized parametric schemes. In particular, we observe the decay of the defects during the Gauß--Newton iterations, which depends on the approximation properties of the underlying neural network architecture. Moreover, long-time properties of the schemes are shown. We present results for the one-dimensional advection equation and heat equation, chosen on purpose as simple test cases in order to show the nontrivial behaviour of the proposed numerical methods. 

\subsubsection*{Related work}
Recently, a substantial amount of research has been devoted to the construction of methods based on the Dirac--Frenkel time-dependent variational principle applied to neural networks. Here, the dynamics of the parameters of the neural network is severely ill-posed \cite[Proposition~3.4]{FLLN24}. Despite this fact, computationally efficient methods have been proposed and numerically studied in the literature. The approach is particularly attractive for high-dimensional partial differential equations, as the methodology is expected to mitigate the curse of dimensionality in many applications \cite{AF24,BPV24,GYZ24}, in particular for quantum many-body systems \cite{CT17,SH20}.  Substantial efforts are currently undertaken to advance and improve the performance of these schemes, see e.g. \cite{SSBP24,KH24,WVP24}. 
Related to the approach in this paper, a parametric midpoint rule has been used in quantum dynamics in \cite{KLPA22} for parametrizations of the wave function by multiple Gaussians and in \cite{GIM22} by neural networks, though without an error analysis.

\subsubsection*{Outline}

After describing the mathematical setting in Section~\ref{sec:setting}, we introduce the proposed regularized parametric approach by the example of the implicit Euler method and the implicit midpoint rule in Section~\ref{sec:intro}. The Gauß--Newton iteration for the regularized parametric schemes is then analyzed in the subsequent Section~\ref{sec:gn-1}, which results in an estimate of the difference of the Gau\ss--Newton iterate and the Newton iterate associated with the non-parametric temporal semi-discretization. Section~\ref{sec:error-bound} then leverages this result, together with classical theory for stiff integrators, to derive error bounds for the Gauß--Newton iterates. The following Sections~\ref{sec:RK}--\ref{sec:RK-error-bound} extend these results to higher-order Runge--Kutta methods. The final Section~\ref{sec:numerics} describes important aspects of the implementation and illustrates and complements the theoretical results by numerical experiments.

\section{Mathematical setting} 
\label{sec:setting}
\subsection{Evolution problem} \label{subsec:ivp}
We study the numerical approximation of an initial-value problem of partial or stiff ordinary differential equations
(with $\dot y = \partial_t y$)
\begin{equation}\label{ivp}
\dot y = f(y), \qquad y(0)=y_0.
\end{equation}
We consider a semilinear equation in a Hilbert space $\calH$ with inner product $\langle\cdot,\cdot\rangle$ and corresponding norm $\|\cdot\|$:
\begin{equation}\label{f}
f(y) = Ay + g(y),
\end{equation}
where the linear operator $A$ is unbounded or may have arbitrarily large norm. We assume that $A$ generates a contraction semigroup, $\| \e^{tA}v \| \le \| v \|$ for all $v\in\calH$, or equivalently by the Lumer--Phillips theorem \cite[Theorem II.3.15]{EN00}, $(I-hA)$ is surjective for all $h>0$ and
the numerical range of $A$ is in the left complex half-plane,
\begin{equation}\label{A}
\Re \langle v, Av \rangle \le 0 \quad\text{for all } v\in D(A),
\end{equation}
with dense domain $D(A)\subset \calH$.
For example, this is satisfied for advection equations on $\R^d$, where $A=a\cdot\nabla_x$ with $a\in\R^d$, for diffusion equations on $\R^d$
with the Laplacian $A=\Delta_x$, and for Schrödinger equations on $\R^d$ with $A=\iu\Delta_x$, in each case considered in the Hilbert space $\calH=L^2(\R^d)$.

We further assume that the nonlinearity $g:\calH\to\calH$ is Lipschitz-continuous, with a moderate Lipschitz constant $L_g$.

\subsection{Evolving parametrization} \label{subsec:par}
We approximate \eqref{ivp}  via a nonlinear parametrization
\begin{equation}\label{Phi}
y(t) \approx u(t)=\Phi(\theta(t)), \qquad 0\le t \le \overline t,
\end{equation}
with time-dependent parameters $\theta(t)$ in a finite-dimensional parameter space $\calQ$ that is equipped with the inner-product norm $\|\cdot\|_\calQ$. The parametrization map $\Phi:\calQ\to\calH$ is assumed twice continuously differentiable with bounded derivatives.

{\it The situation of interest here is when the derivative $\Phi'(\theta)$ can have arbitrarily small singular values and possibly varying rank.} 

This is a geometrically irregular situation: the image $\calM=\{\Phi(\theta)\,:\, \theta\in \calQ\}\subset\calH$ can then be a manifold with arbitrarily large curvature or no manifold at all.

A widely used approach for determining the evolving parameters is known as the Dirac--Frenkel time-dependent variational principle in quantum physics and chemistry; see e.g. \cite{Dir30,He76,MMC92,Hae11,CT17} for applications in these fields and \cite{KraS81,Lub08} for dynamical, geometric and approximation aspects. 
In this variational approach, the time derivative $\dot \theta(t)$ of the parameters is determined by requiring that at $u(t)=\Phi(\theta(t))$, 
the derivative $\dot u(t)=\Phi'(\theta(t))\dot\theta(t)$ has minimal defect:
\begin{equation}\label{tdvp}
\| \dot u(t) - f(u(t)) \| \to\min,\ \text{ i.e.}\ \ 
\| \Phi'(\theta(t))\dot\theta(t) - f(\Phi(\theta(t))) \| \to\min.
\end{equation}
This is a linear least squares problem for determining $\dot \theta$ as a function of $\theta$, thus yielding a differential equation for $\theta$. However, in our situation of interest this problem is ill-posed.
As in \cite{FLLN24}, we consider the regularized evolution problem with a regularization parameter $\eps>0$,
\begin{equation}\label{tdvp-reg}
\delta(t)^2 := \| \dot u(t) - f(u(t)) \|^2 + \eps^2 \|\dot \theta(t) \|^2 \to\min.
\end{equation}
As is shown in \cite{FLLN24}, this problem is still severely ill-posed for nonlinear parametri\-zations, but it is well-posed for $u$ up to the defect size $\delta$ (though not for the parameters $\theta$). This has been formulated and proved under conditions on the differential equation that include the semilinear setting of Section~\ref{subsec:ivp}.


\section{Parametric implicit Euler and midpoint methods} \label{sec:intro}

\subsection{Reference point: the implicit Euler method}
Consider time discretization by the implicit Euler method with step size $h>0$,
\begin{equation}\label{implicit-euler}
\frac{y_{n+1} - y_n} h = f(y_{n+1}),
\end{equation}
where $y_n\in\calH$ is to approximate $y(t_n)$ at $t_n=t_0+nh\le \bar t$.
In every time step, the nonlinear equation for $y_{n+1}$ is solved approximately by a few iterations of a modified Newton method: $y_{n+1}^0=y_n$ and for $k=0,1,\dots,K$,
\begin{align}\nonumber
&y_{n+1}^{k+1} = y_{n+1}^k + \Delta y_{n+1}^k
\quad\text{ with }
\\[1mm]
\label{newton}
& (I-hJ_n)\Delta y_{n+1}^k = 
- \bigl(y_{n+1}^k- y_n - h f(y_{n+1}^k)\bigr).
\end{align}
The matrix $J_n$ is chosen as an approximation to the Jacobian matrix $f'(y_n)$.
In the semilinear setting of this paper (see Section~\ref{subsec:ivp}) we simply take $J_n=A$.

The implicit Euler method described here is an abstract method in an infinite-dimensional or prohibitively high-dimensional space~$\calH$. It needs yet to be approximated by a computable parametrization in a lower-dimensional space~$\calQ$. Nevertheless, known properties of this abstract implicit Euler method will be useful to derive theoretical properties of the computational parametrized method.



\subsection{Regularized parametric implicit Euler method}
In the case of a partial differential equation (posed in an infinite-dimensional Hilbert space), \eqref{implicit-euler}-\eqref{newton} yields a semi-discrete scheme that needs to be further approximated, in our case by a nonlinear parametrization. The same necessity for a parametric approximation arises for high-dimensional ordinary differential equations, e.g. for those arising from spatial semi-discretizations of evolutionary partial differential equations or for quantum spin systems (see e.g. \cite{Hae11} and references therein) or for the chemical master equation (see e.g. \cite{Gil92,JH08} and references therein).

We determine the parametric approximation $u_{n+1}=\Phi(\theta_{n+1}) \approx y_{n+1}$ to \eqref{implicit-euler} by requiring that $\theta_{n+1}$
solve a nonlinear least-squares problem that discretizes \eqref{tdvp-reg},
\begin{equation}\label{par-ie-1}
\biggl\| \frac{u_{n+1}-u_n}h - f(u_{n+1}) \biggr\|^2 
+ \eps^2 \,\biggl\| \frac{\theta_{n+1}-\theta_n}h \biggr\|_\calQ^2 
\to \min.
\end{equation}
This is solved approximately by a few iterations of a Gau\ss--Newton method that is formulated and analysed in Section~\ref{sec:gn-1}. The iteration requires first derivatives of  $\Phi$ and $f$ at the previous values  $\theta_n$ and $u_n$, respectively, or instead of $f'(u_n)$ just the fixed operator $A$ in our semilinear setting.

\subsection{Variants}
In a different approach, we choose $u_{n+1}=\Phi(\theta_{n+1}) \approx y_{n+1}$ with
$$
\theta_{n+1} = \theta_n + h \dot \theta_{n+1},
$$
where $\dot \theta_{n+1}\in\calQ$ is determined as a solution of the regularized nonlinear least-squares problem,
here written with $\dot u_{n+1}=\Phi'(\theta_{n+1})\dot \theta_{n+1}$,
\begin{equation}\label{par-ie-2}
\| \dot u_{n+1} - f(u_{n+1}) \|^2 + \eps^2 \| \dot \theta_{n+1} \|_\calQ^2 
\to \min.
\end{equation}
This is again solved approximately by a Gau\ss--Newton iteration.

We note that both methods coincide for a linear parametrization $\Phi$, which is, however, not the case of interest here.
For the explicit Euler method, the analogue of \eqref{par-ie-2} requires the solution of a linear least-squares problem, whereas the analogue of \eqref{par-ie-1} has a nonlinear least-squares problem. For the implicit Euler method, both variants involve nonlinear least-squares problems. Both methods have similar theoretical properties. In our numerical experiments, the method based on \eqref{par-ie-1} turned out to be at least as good as the method based on \eqref{par-ie-2} in all tests and more robust for some problems.

The minimization \eqref{par-ie-2} is equivalent to using the implicit Euler method to discretize the differential equation that is obtained for the parameters $\theta(t)\in\calQ$ by the regularized dynamical approximation \eqref{tdvp-reg},
that is, by the regularized normal equations with the adjoint linear map $\Phi'(\theta)^*:\calH\to\calQ$
(omitting the argument $t$), 
$$
(\Phi'(\theta)^*\Phi'(\theta) + \eps^2 I) \dot \theta = 
\Phi'(\theta)^* f(\Phi(\theta)).
$$
Solving the nonlinear equations arising from the implicit Euler discretization of this differential equation for $\theta$ by a Newton-type method requires, however, computing second derivatives of the parametrization map $\Phi$, in contrast to the Gau\ss--Newton method mentioned before and used in the following sections.

\subsection{Implicit midpoint rule}
An analogous parametric version can be given for the midpoint rule
\begin{equation}\label{mpr}
\frac{y_{n+1} - y_n} h = f\Bigl(\frac{y_{n+1}+y_n}2\Bigr) .
\end{equation}
Instead of \eqref{par-ie-1} we then take, again with $u_n=\Phi(\theta_n)$,
\begin{equation}\label{par-mp}
 \biggl\| \frac{u_{n+1}-u_n}h - f\Bigl(\frac{u_{n+1}+u_n}2\Bigr) \biggr\|^2  
+ \eps^2 \, \biggl\| \frac{\theta_{n+1}-\theta_n}h \biggr\|_\calQ^2 
\to \min.
\end{equation}
This nonlinear least-squares problem for $\theta_{n+1}$ is solved approximately by a few iterations of a Gau\ss--Newton method as described in Section~\ref{sec:gn-1}.

\section{Regularized Gau\ss--Newton iteration}
\label{sec:gn-1}
For ease of notation we set the step number to $n=0$ in this section, so that we aim to compute the parameters $\theta_1$ of $u_1=\Phi(\theta_1)$ starting from $u_0=\Phi(\theta_0)$. 
To approximate \eqref{par-ie-1}, we can set up a Gau\ss--Newton iteration where we compute $u_1^{k+1}=\Phi(\theta_1^{k+1})$ from
$$
\theta_1^{k+1} = \theta_1^k + \Delta \theta_1^k, \qquad k=0,1,\ldots,K,
$$
starting from $\theta_1^0=\theta_0$, by solving the regularized linear least-squares problems
$$
\| \bigl(I -hJ_0\bigr) \Phi'(\theta_0) \Delta \theta_1^k/h + r^k \|^2 + \eps^2 \| \Delta \theta_1^k/h + \sigma^k \|_\calQ^2 \to \min
$$
with $r^k = (u_1^k - u_0)/h - f(u_1^k)$ for $u_1^k=\Phi(\theta_1^k)$ and $\sigma^k=(\theta_1^k-\theta_0)/h$. Note that
$\Delta \theta_1^k + h\sigma^k=\theta_1^{k+1}-\theta_0$.
The matrix $J_0$ is chosen as an approximation to $f'(y_0)$.
In our semilinear setting (see Section~\ref{subsec:ivp}) we take $J_0=A$.

Our error analysis indicates, however, that an iteration with potentially better properties is obtained with a modified regularization:  
\begin{align} \nonumber
\bigl(\delta^k\bigr)^2 := \ 
&\| \bigl(I -hJ_0\bigr) \Phi'(\theta_0) \Delta \theta_1^k/h + r^k \|^2 \\
&+ \tfrac12\eps^2 \| \Delta \theta_1^k/h + \sigma^k \|_\calQ^2 
+ \eps^2 \| \Delta \theta_1^k/h \|_\calQ^2 
\to \min.
\label{gn-1}
\end{align}
The following theorem compares the Gau\ss--Newton iterate $u_1^{k+1}=\Phi(\theta_1^{k+1})$ with the Newton iterate $y_1^{k+1}$ of the (abstract, non-parametrized) implicit Euler method \eqref{newton} for the same starting value $y_1^0=y_0=u_0=u_1^0$.

\begin{theorem}[Iteration error] \label{thm:gn-err}
In the setting of Section~\ref{sec:setting} and under a restriction between the stepsize $h$ and the regularization parameter $\eps$,
\begin{equation}\label{h-eps}
h\delta^{k} \le c\eps^2\qquad\text{for all $k$},
\end{equation}
the regularized Gau\ss--Newton iteration \eqref{gn-1} with $J_0=A$ yields iterates $u_1^{k+1}=\Phi(\theta_1^{k+1})$ that deviate from the
Newton iterates $y_1^{k+1}\in\calH$ of \eqref{newton} by
\begin{equation}\label{gn-err}
\| u_1^{k+1} - y_1^{k+1}\| \le  Ch (\delta^k+\rho\delta^{k-1}+\ldots +\rho^k\delta^0),
\end{equation}
where $C=1+\beta c$ with the bound $\beta$ of the second derivative of $\Phi$ in a neighbourhood of $\theta_0$, 
and $\rho=hL_g$ with the Lipschitz constant $L_g$ of $g$ in \eqref{f}.
\end{theorem}

\begin{proof}
Let
$$
d^k:= (I -hA) \Phi'(\theta_0) \Delta \theta_1^k/h + r^k,
$$
for which \eqref{gn-1} yields 
$$
\| d^k \|^2 +  \tfrac12\eps^2 \| (\theta_1^{k+1}-\theta_0)/h  \|_\calQ^2 
+ \eps^2 \| \Delta\theta_1^{k}/h  \|_\calQ^2 = (\delta^k)^2,
$$
so that 
$$
\| d^k \| \le \delta^k \quad\text{ and }\quad
\tfrac12 \| \theta_1^{k+1}-\theta_0 \|_\calQ^2+ \| \Delta\theta_1^{k}  \|_\calQ^2 \le \Bigl(\frac{h\delta^k}\eps\Bigr)^2.
$$
We further introduce a term that accounts for the nonlinearity of the parametri\-zation map $\Phi$,
\begin{align*}
w^k &:= \frac{\Phi(\theta_1^{k+1})-\Phi(\theta_1^k)}h - \Phi'(\theta_0)\frac{\Delta \theta_1^k}h
\\
&\ =  \biggl(\frac{\Phi(\theta_1^{k+1})-\Phi(\theta_1^k)}h - \Phi'(\theta_1^{k+1})\frac{\Delta \theta_1^k}h\biggr) 
+ \Bigl(  \Phi'(\theta_1^{k+1}) -  \Phi'(\theta_0) \Bigr) \frac{\Delta \theta_1^k}h.
\end{align*}
With the bound $\beta$ of the second derivative of $\Phi$, this is bounded by
\begin{align*}
 \| hw^k \| &\le \tfrac12 \beta\, \| {\Delta \theta_1^k} \|_\calQ^2  
 + \beta \, \| \theta_1^{k+1} - \theta_0 \|_\calQ \,\|{\Delta \theta_1^k} \|_\calQ 
\\[1mm]
&\le \beta \bigl( \tfrac12 \| \theta_1^{k+1}-\theta_0 \|_\calQ^2+ \| \Delta\theta_1^{k}  \|_\calQ^2 \bigr)
\\
 &\le \beta\, \Bigl(\frac{h\delta^k}\eps\Bigr)^2.
\end{align*}
Under the stepsize restriction \eqref{h-eps}
we thus obtain the bound
$$
\| h w^k \| \le \beta c\, h\delta^k.
$$
Moreover, we let $\Delta u_1^{k}= u_1^{k+1}-u_1^k$.
With these quantities, we have the equation
$$
(I-hA) \Delta u_1^{k} - (I-hA) h w^k + hr^k = h d^k.
$$
On the other hand, for the Newton method \eqref{newton} we have the equation
$$
(I-hA) \Delta y_1^{k}  + h\widehat r^k = 0,
$$
where $h\widehat r^k = y_{1}^k- y_0 - h f(y_{1}^k)$.
For the errors $e^k=u_1^k-y_1^k$ and their increments $\Delta e^k=e^{k+1}-e^k$ we thus obtain
$$
(I-hA) \Delta e^{k} - (I-hA)h w^k + (I-hA)e^k + h(g(u_1^k)-g(y_1^k)) = h d^k
$$
or equivalently
$$
(I-hA) e^{k+1} = h d^k + (I-hA)h w^k - h(g(u_1^k)-g(y_1^k)).
$$
Since condition \eqref{A} implies $\| (I-hA)^{-1}\|\le 1$,  this yields
\begin{align*}
\| e^{k+1} \| &\le h\delta^k + \beta c\, h\delta^k + hL_g\| e^k \| 
= (1+\beta c)h\delta^k + \rho \|e^k\|, 
\end{align*}
and the stated result follows. 
\qed
\end{proof}

\begin{remark}[Stepsize vs. regularization restriction]
Condition \eqref{h-eps} is the same as for regularized parametric explicit methods in the nonstiff case of (moderately) Lipschitz-bounded $f$ \cite{FLLN24}. Analogously to the nonstiff case, the step size restriction is not necessary when $\delta^{k}\le\widetilde{C}\varepsilon$, for some constant $\widetilde{C}$ that may enter into the error constant. Numerical evidence (see Figure~\ref{fig:eps-to-delta}) suggests that such a condition often holds in practical computations for $\eps$ larger than the minimal defect size attained without (or very small) regularization. In this case, the term $hw^k$, which measures the effect of the nonlinearity of $\Phi$, is bounded by $O(h^2)$ without a restriction of $h$ by $\eps$. This $O(h^2)$ bound suffices for a first-order global error bound (see the next section) but is not sufficient for obtaining higher order of approximation with higher-order methods such as the implicit midpoint rule or multi-stage implicit Runge--Kutta methods.
\end{remark}

\begin{remark}[Variants of regularization]
    Variants where only one of the two regularisation terms in \eqref{gn-1} is taken, can also be used, but with a different stepsize vs. regularization restriction. If only the last term is taken, then we have 
    $\|\Delta \theta^{k}\|_\calQ \le h\delta^k/\eps$
    and we need to bound 
    $$
    \| \theta_1^k-\theta_0 \|_\calQ \le \|\Delta \theta^{k-1}\|_\calQ + \ldots + \|\Delta \theta^{0}\|_\calQ\le
    h (\delta^{k-1}+\ldots+\delta^0)/\eps.
    $$
    This then leads to the same error bound as in Theorem~\ref{thm:gn-err} under the restriction
    $$
     h (\delta^{k}+\ldots+\delta^0) \le c\eps^2.
    $$
    On the other hand, if only the first regularization term is taken (without the factor $\tfrac12$), then we have 
    $\|\theta^{k+1}-\theta_0\|_\calQ \le h\delta^k/\eps$ and we need to bound
    $$
    \|\Delta \theta^{k}\|_\calQ \le \|\theta^{k+1}-\theta_0\|_\calQ + \|\theta^{k}-\theta_0\|_\calQ \le
    h(\delta^k+\delta^{k-1})/\eps.
    $$
    This leads to a similar error bound as in Theorem~\ref{thm:gn-err} (with a different constant~$C$) under the restriction
    $$
     h (\delta^{k}+\delta^{k-1}) \le c\eps^2.
    $$
    Note, however, that the defect sizes $\delta^k$ are different for the different variants, which precludes
    a direct theoretical comparison.
    
    The effects of different choices of regularization are reminiscent of the difference between Levenberg--Marquardt vs. Bakushinskii versions of regularized Gau\ss--Newton iterations; see the discussion of formulas (15) vs. (16) in \cite{Kal97}.
\end{remark}

\begin{remark}[Adaptive regularization parameter]
    In numerical experiments we typically observe that the defect sizes $\delta^k$ initially decay rapidly with growing $k$ until they saturate at a level that depends on $\eps$. This suggests using varying, decaying regularization parameters $\eps^k$. A simple heuristic strategy, based on the assumption that $\delta^k \le \delta^{k-1}$ (which can be checked {\it a posteriori}), is to choose $\eps^k$ such that $h\delta^{k-1}= c (\eps^k)^2$ for some fixed number $c$.
    In our numerical experiments we do not pursue this further but will keep $\eps$ fixed over a time step, since we typically have good starting values and need only few iterations. However, we let $\eps$ vary from one time step to the next.
\end{remark}

\begin{remark}[Parametric implicit midpoint rule]
 The analogous regularized Gau\ss--Newton method applied to the parametric implicit midpoint rule \eqref{mpr} reads
\begin{align} \nonumber
\bigl(\delta^k\bigr)^2 := \ 
&\| \bigl(I -\tfrac12 hJ_0\bigr) \Phi'(\theta_0) \Delta \theta_1^k/h + r^k \|^2 
\\
&+ \tfrac12\eps^2 \| \Delta \theta_1^k/h + \sigma^k \|_\calQ^2 + \eps^2 \| \Delta \theta_1^k/h \|_\calQ^2 \to \min
\label{gn-mpr}
\end{align}
with $r^k = (u_1^k - u_0)/h - f(\tfrac12(u_1^k+u_0))$ for $u_1^k=\Phi(\theta_1^k)$ and $\sigma^k=(\theta_1^k-\theta_0)/h$.
For this iteration, there is an analogous result to Theorem~\ref{thm:gn-err}, which is proved in the same way.
\end{remark}

\section{Error bound}
\label{sec:error-bound}
In this section we analyse the error of the method that uses a fixed number of $K$ iterations of the regularized Gau\ss--Newton method \eqref{gn-1} in every time step: starting from $u_n=\Phi(\theta_n)$ in the $(n+1)$-st time step, the method computes $u_{n+1}=u_{n+1}^K$ by $K$ regularized Gau\ss--Newton iterations. For ease of presentation, we take the same step size $h$ in every time step, the same regularization parameter $\eps$ in every iteration, and the same number $K$ of iterations. The extension to varying $h$, $\eps$ and $K$ presents only notational difficulties for the theory, but simultaneous adaptivity for $h$, $\eps$ and $K$ in the actual algorithm is a formidable task.

\subsection{Preparation: Bounds for the non-parametric implicit Euler method}

We collect a few estimates that will be used later.

\begin{lemma} [Local error]
\label{lem:ie-newton}
    Let $y_1$ be the result of an implicit Euler step \eqref{implicit-euler} starting from $y_0\in\calH$, and let $y_1^k$ be the result after $k$ modified Newton iterations \eqref{newton} with $J_0=A$ and with the starting iterate $y_1^0=y_0$. Assume that the step size is bounded by $hL_g\le\rho <1$, where $L_g$ is the Lipschitz constant of $g$. Then, 
    \begin{align} \label{newton-err}
        \| y_1^{k+1} - y_1 \| &\le \rho \,\| y_1^k -y_1 \| 
    \\ \label{newton-init}
    \text{and}\quad
        \| y_1^0- y_1 \| = \| y_0-y_1 \| &\le \frac1{1-\rho}\,\bigl( \| y_0 \| + h \|g(y_0)\|\bigr).
    \end{align}
    (If $y_0\in D(A)$ with $\|A y_0\|\le \alpha$, then the bound \eqref{newton-init} has $h\alpha$ in place of $\|y_0\|$.)\\
    Furthermore, the local error of the implicit Euler method is bounded by
      \begin{equation} \label{loc-err-ie}
        \| y_1- y(t_1) \| \le h^2\,(1-\rho)^{-1} \,\tfrac12 \max_{t\in[t_0,t_1]}\| \ddot y(t) \|.
    \end{equation}  
\end{lemma}

\begin{proof}
(i)    We rewrite \eqref{implicit-euler} and \eqref{newton} as
    \begin{align*}
        (I-hA) y_1^{k+1} &= y_0 + hg(y_1^k) \\
        (I-hA) y_1       &= y_0 + hg(y_1).
    \end{align*}
    Taking the difference of the two equations and then the inner product with $y_1^{k+1} - y_1$ yields, using \eqref{A} on the left-hand side and the Cauchy--Schwarz inequality on the right-hand side,
    $$
    \| y_1^{k+1} - y_1 \|^2 \le h \, \| y_1^{k+1} - y_1 \|\: \| g(y_1^k) - g(y_1) \|.
    $$
    With the Lipschitz condition for $g$, this yields \eqref{newton-err}.

    (ii) We rewrite \eqref{implicit-euler} as
    $$
    y_1 -y_0 = hA(I-hA)^{-1} y_0+ (I-hA)^{-1} h\bigl(g(y_1)-g(y_0)\bigr) + (I-hA)^{-1} h g(y_0)
    $$
    where we used $(I-hA)^{-1}-I= hA(I-hA)^{-1}$. By \eqref{A} and a theorem of von Neumann \cite{vNeu} (see also Crouzeix's elegant proof in \cite[Section IV.11]{HW91}) we have
    $$
    \| hA(I-hA)^{-1} \| \le \sup_{\Re z \le 0} \, \Bigl| \frac{z}{1-z} \Bigr| = 1.
    $$
    Since also $\|(I-hA)^{-1} \| \le 1$, we obtain \eqref{newton-init}. The bound in the case of $y_0\in D(A)$ is then immediate.

    (iii) The local error bound \eqref{loc-err-ie} is well known; see e.g. \cite{HL15}. It is obtained by inserting the exact solution into the numerical scheme,
    \begin{equation} \label{ie-defect}
    \frac{y(t_{1}) - y(t_0)} h = f(y(t_1)) + d
    \end{equation}
    with the defect $d$.
    Since $f(y(t_1))=y'(t_1)$, Taylor expansion at $t_1$ shows that $d$ is bounded by
    $$
    \| d \| \le  \tfrac12 h^2\,\max_{t\in[t_0,t_1]}\| \ddot y(t) \|. 
    $$
    The bound \eqref{loc-err-ie} then follows by subtracting \eqref{ie-defect} from \eqref{implicit-euler}, taking the inner product with $y_1-y(t_1)$ and using \eqref{A} and the Lipschitz continuity of $g$ as 
    in~(i).
    \qed
\end{proof}

The following basic stability inequality is also well known; see e.g. \cite{HL15}. The proof is essentially the same as (i) above.

\begin{lemma}[Stable error propagation by the implicit Euler method]
 \label{lem:stable-ie}
    Let $y_{1}$ and $\wt y_{1}$ be the results of a step of the (non-parametric) implicit Euler method \eqref{implicit-euler} starting from $y_0$ and $\wt y_0$, respectively. Then, provided that ${hL_g\le \tfrac12}$,
    \begin{equation} \label{stable-ie}
    \| y_{1} - \wt y_{1} \| \le \frac 1{1-hL_g} \, \| y_{0} - \wt y_{0} \| \le \e^{2L_g h} \, \| y_{0} - \wt y_{0} \|.
    \end{equation}
\end{lemma}

We deduce a global error bound by the standard argument of Lady Windermere's fan \cite[Theorem I.7.3]{HNW}, which is visualized in Figure~\ref{fig:Windermere-yn-y(tn)}. The necessary ingredients are provided by the local error bound \eqref{loc-err-ie} and the stability bound \eqref{stable-ie}, which yield the global error bound on an interval of length $T=\bar t-t_0$, 
\begin{equation}\label{glob-err-ie}
\| y_n - y(t_n) \| \le h\,\frac{e^{2L_g T} -1}{2L_g}\, \tfrac12 \max_{t\in[t_0,\bar t]}\| \ddot y(t) \|,
\qquad 0\le nh \le T.
\end{equation}
\begin{figure}
    \centering
    \includegraphics[width=0.9\linewidth]{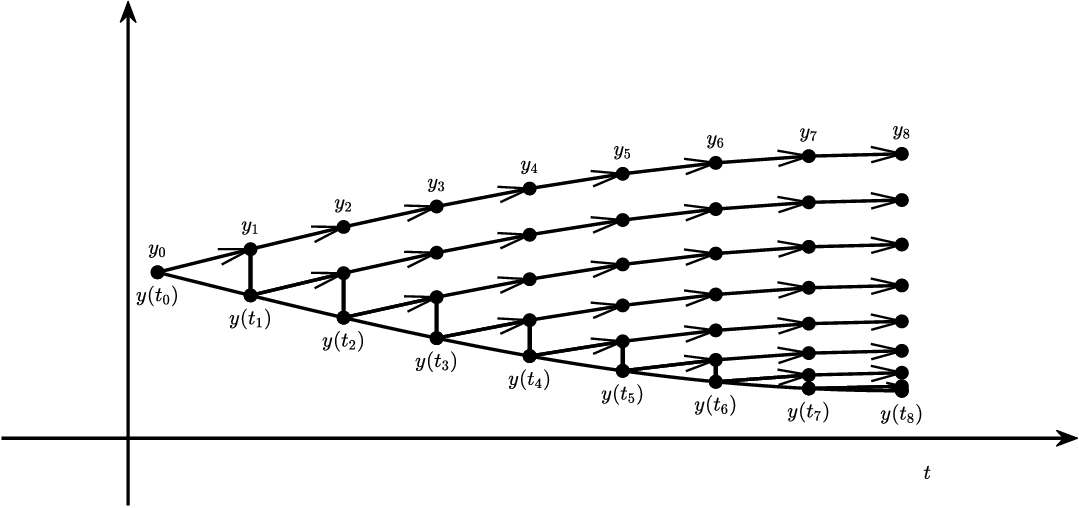}
    \caption{Lady Windermere's fan for the global error $y_n-y(t_n)$ of the non-parametric (semi-discrete) implicit Euler method \eqref{implicit-euler}. Each arrow symbolizes a step of the implicit Euler method. The vertical lines illustrate local errors, which are then propagated by the implicit Euler method.
    The global error is the sum of the propagated local errors.
    }
     \label{fig:Windermere-yn-y(tn)}
\end{figure}
\subsection{Local error}
The local error of the parametric method, i.e., the difference between the numerical solution $u_1=u_1^K$ after one step and the exact solution $y(t_1)$ of \eqref{ivp} with the same initial value $y(t_0)=u_0=\Phi(\theta_0)$, is split as
$$
u_1-y(t_1) = \bigl( u_1-y_1 \bigr) + \bigl( y_1 - y(t_1)\bigr), 
$$
where the last term is the local error of the non-parametric implicit Euler discretization, which is bounded by \eqref{loc-err-ie}.
The first term on the right-hand side accounts for the error contribution due to the parametrization and the regularized Gau\ss--Newton iteration. This term is bounded by combining Theorem~\ref{thm:gn-err} and Lemma~\ref{lem:ie-newton}. Again for $hL_g= \rho<1$,
\begin{align}\label{loc-err}
\| u_1 - y_1 \| &\le \| u_1^K - y_1^K \| + \| y_1^K - y_1 \| 
\nonumber
\\
&\le C h (\delta_1^K+\rho\delta_1^{K-1}+\ldots +\rho^K\delta_1^0) +  \rho^K  \, \| u_0-y_1 \|
\\
&\le C h (\delta_1^{K}+\rho\delta_1^{K-1}+\ldots +\rho^K\delta_1^0) +  \rho^K  \, \frac1{1-\rho}\bigl(\| u_0 \| + h\| g(u_0) \|\bigr)
\nonumber
\end{align}
where we now added the subscript $1$ to $\delta^k$ in \eqref{gn-1}.

\subsection{Global error} \label{subsec:ge-ie}
We now turn to the error after $n$ steps, for $n\ge 0$ with $t_n=t_0+nh \le \bar t$.
We apply the iteration \eqref{gn-1} in the $n$th time step, with $J_n=A$ and starting from $u_{n}=\Phi(\theta_n)$.
We then get defect sizes $\delta_{n+1}^k$ for $k=0,\dots,K-1$ in computing $u_{n+1}=u_{n+1}^K= \Phi(\theta_{n+1}^K)$. We define, again with $\rho=hL_g\le \tfrac12$,
\begin{align*}
\delta_n &:= \delta_n^{K}+\rho\delta_n^{K-1}+\ldots +\rho^K\delta_n^0
\\
\eta_n &:= h^{K-1}\,2L_g^{K-1}\bigl(\| u_n \| + h\|g(u_n)\|\bigr).
\end{align*}

\begin{theorem} [Global error bound]  \label{thm:ie-err}
In the setting of Section~\ref{sec:setting} and under the conditions $hL_g\le\tfrac12$ and
\begin{equation}\label{h-eps-global}
h (\delta_n^{K-1}+\ldots+\delta_n^0) \le c\eps^2, \qquad 0\le nh \le T =\bar t - t_0,
\end{equation} 
we have with $\delta=\max\delta_n$ and $\eta=\max \eta_n$
\begin{equation}
    \| u_n - y_n \| \le  C_1\delta + C_2\eta, \qquad 0\le nh \le T ,   
\end{equation}
where $C_1= C\bigl(e^{2L_g T} -1\bigr)/(2L_g)$ and $C_2 = e^{2L_g T} -1$,  with $C=1+\beta c$ of Theorem~\ref{thm:gn-err} and $\beta$ a bound of the second derivative of $\Phi$ in a neighbourhood of the parameter trajectory.

Combined with the error bound \eqref{glob-err-ie} of the non-parametric implicit Euler method, this yields a global error bound (for $K\ge 2$ iterations per time step)
$$
\|u_n-y(t_n)\|=O(h+\delta).
$$
\end{theorem}

    \begin{figure}[h!]
    \centering
    \includegraphics[width=0.9\linewidth]{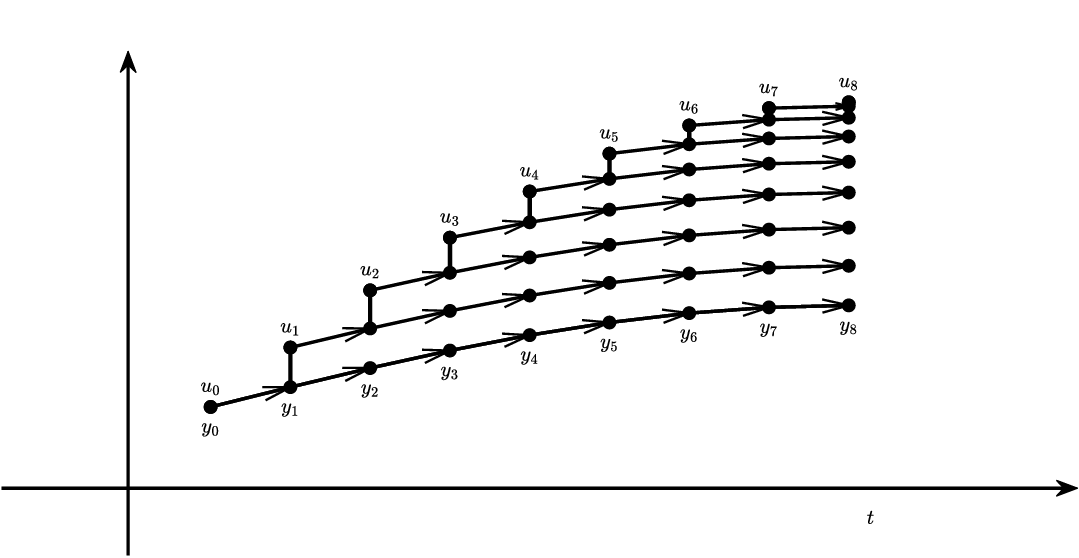}
    \caption{Lady Windermere's fan for the estimation of the global difference $\|u_n-y_n\|$, which measures the effect of the regularized parametrization on the implicit time integration method. Arrows denote a single time step of the underlying (semi-discrete) time integration scheme.
    }
    \label{fig:Windermere-un-yn}
\end{figure}

\begin{proof}
The local error bounds of Theorem~\ref{thm:gn-err} and Lemma~\ref{lem:ie-newton} in combination with the stable error propagation \eqref{stable-ie} allow us to apply the standard argument of Lady Windermere's fan, which is visualized in Figure~\ref{fig:Windermere-un-yn}. Note that $u_n-y_n$ is the sum of the propagated local errors. The local errors are bounded in \eqref{loc-err} and are propagated in a stable way according to Lemma~\ref{lem:stable-ie}. This yields the stated result.
\qed
\end{proof}

\begin{remark}[Parameter deviation]\label{rem:deviation-theta}
The constant $\beta$, which bounds the second derivative of $\Phi$, enters in the local error estimate and thus in the global error bound. This bound on the second derivative does not necessarily need to be global, but is required to hold in a neighborhood of the values $\theta_n$. The approximations $\theta_n$ remain, under modeling assumptions, in a neighbourhood of the initial parameter values, which is observed by computing
\begin{align*}
\|\theta_n - \theta_0\|_{\calQ} \le 
\sum_{j=1}^n\sum_{k=0}^{K-1}\|\Delta \theta_j^k \|_{\calQ}
\le C n \dfrac{\delta h}{\varepsilon}
=Ct_n\dfrac{\delta}{\varepsilon}.
\end{align*}
If the ratio $\delta/\varepsilon$ is bounded, which we observed frequently in our numerical experiments, as shown in Figure~\ref{fig:eps-to-delta}, we therefore obtain that the parameters only deviate linearly from the initial parametrization within a prescribed neighbourhood of $\theta_0$. Under the step size restriction \eqref{h-eps-global} of Theorem~\ref{thm:ie-err}, we further obtain the counterintuitive bound
\begin{align*}
\|\theta_n - \theta_0\|_{\calQ} 
\le C n\varepsilon,
\end{align*}
which indicates a moderate deviation of the parameters from their initialization.
\end{remark}

\begin{remark}[Midpoint rule]
An analogous result holds true for the regularized parametric midpoint rule \eqref{gn-mpr}. By the same arguments we obtain the second-order global error bound (for $K\ge 3$ iterations per time step)
$$
\|u_n-y(t_n)\|=O(h^2+\delta).
$$    
\end{remark}

\begin{remark}
We mention the related result of \cite[Proposition~6]{ZCVP24}, which bounds the global error of the idealized schemes \eqref{par-ie-1} and \eqref{par-mp} without regularization ($\eps=0$) in terms of the defects therein and the local errors of the underlying non-parametric Euler or midpoint method. In contrast, here we study the effect of the computationally accessible regularized Gau\ss--Newton iteration \eqref{gn-1} on the global error.
\end{remark}

%
%

\section{Parametric implicit Runge--Kutta methods}
\label{sec:RK}
\def\doubleone{ 1 \kern-.224em \hbox{\rm l}}
A step of an $s$-stage implicit Runge--Kutta method applied to \eqref{ivp} with step size $h>0$ is given by
\begin{align}\nonumber
    y_1 &= \displaystyle y_0 + h \sum_{i=1}^s b_i f(Y_i) 
    \\
    Y_i &= \displaystyle y_0 + h \sum_{j=1}^s a_{ij}f(Y_j) ,\quad\ i=1,\ldots ,s , \label{irk}
\end{align}
where $y_1\approx y(t_1)$, and $Y_1,\ldots,Y_s$ are the internal stages. For stiffly accurate Runge--Kutta methods, like Radau IIA, we have $y_1 =Y_s$. For methods with invertible coefficient matrix ${\calA} = (a_{ij})_{i,j=1}^s$, like Gauss methods,
we eliminate $f(Y_i)$ in the first line to obtain
\begin{equation} \label{y1}
    y_1 = y_0 + \sum_{i=1}^s w_i (Y_i - y_0)
\end{equation}
with $w_i$ given by $\sum_{i=1}^s w_i a_{ij} = b_j$.

\subsection{Modified Newton iteration for the nonlinear Runge--Kutta equations}
\label{subsec:newt-irk}

The nonlinear system for $Y= (Y_i)_{i=1}^s$ can be solved by modified Newton iterations.
Using the notation $Y^k = (Y_i^k)_{i=1}^s$ for the $k$th iterate and $F_Y^k = \bigl( f(Y_i^k)\bigr)_{i=1}^s$,
the increment $\Delta Y^k = Y^{k+1} - Y^k$ is computed from the linear system
\[
\bigl( I - h \calA\otimes J_0 \bigr) \Delta Y^k = - (Y^k -\doubleone \otimes y_0 ) + h (\calA \otimes I) F_Y^k =:  h R^k,
\]
where $J_0$ is again an approximation to $f'(y_0)$, which we take as $J_0=A$ in the semilinear setting of Section~\ref{sec:setting}. Proceeding as in \cite[Section IV.8]{HW91},
we multiply this equation by $\calA^{-1}\otimes I$ and diagonalize the coefficient matrix $\calA$ as
$$
\calT^{-1}\calA^{-1} \calT = \Lambda = {\rm diag} (\lambda_i)_{i=1}^s.
$$
We normalize $\calT$ by requiring
$\|\calT\|_2=1$.
For 
$\Delta \hat Y^k = (\calT^{-1}\otimes I) \Delta Y^k$ and
$ \hat R^k = ( \calT^{-1}\otimes I)  R^k$
this gives
\[
\Bigl( \Lambda \otimes I - h I \otimes J_0 \Bigr) \frac{\Delta \hat Y^k}{h} = (\Lambda \otimes I ) \hat R^k ,
\] 
which is taken together with
$$
 Y^{k+1} = (\calT\otimes I) \hat Y^{k+1} ,\quad \hat Y^{k+1}= \hat Y^k + \Delta \hat Y^k .
$$
The linear system then decouples into $s$ independent linear equations
\begin{equation}\label{newton-rk}
(\lambda_i I - hJ_0) \frac{\Delta \hat Y_i^k}{h} =\lambda_i \hat R_i^k, \qquad i=1,\dots,s.
\end{equation}


\subsection{Regularized Gau\ss--Newton iteration for the parametric version}

We aim to approximate the iterates of the stage vector $Y^k$ by parametrized values 
$U^k = \Phi (\Theta^k ) = \bigl(\Phi(\Theta_i^k)\bigr)_{i=1}^s$.
For this we consider the iteration in the parameter space
\[
\Theta^{k+1} = \Theta^k + \Delta \Theta^k, \quad \Delta \Theta^k =  (\calT\otimes I) \Delta \hat \Theta^k.
\]
We choose $\Delta \hat \Theta^k$ as the solution of the regularized linear least-squares problem, cf. \eqref{gn-1},
\begin{align*}
(\delta ^k)^2 := \  
&\Big\| \Bigl( \Lambda \otimes I - h I \otimes J_0 \Bigr)\bigl( I\otimes \Phi' (\theta_0)\bigr) \frac{\Delta \hat \Theta^k}{h}
- ( \Lambda \otimes I) \hat S^k \Big\|^2 
\\
&+ \tfrac12\eps^2   \Big\|  \frac{\Delta \hat \Theta^k}{h} + \hat\Sigma^k \Big\|_{\calQ^s}^2 
+ \eps^2   \Big\|  \frac{\Delta \hat \Theta^k}{h} \Big\|_{\calQ^s}^2  \to \min ,
\end{align*}
where
$\hat S^k = ( \calT^{-1}\otimes I )S^k$ and $hS^k = - (U^k - \doubleone \otimes u_0) +
h (\calA \otimes I) F_U^k$
with $F_U^k = \bigl( f(U_i^k ) \bigr)_{i=1}^s$ and $u_0 = \Phi (\theta_0)$. Furthermore,
$\hat \Sigma^k = ( \calT^{-1}\otimes I )\Sigma^k$ and $h\Sigma^k = \Theta^k - \doubleone \otimes \theta_0$.

This least-squares problem for $\Delta \hat \Theta^k=(\Delta \hat \Theta_i^k)_{i=1}^s$ decouples into $s$ independent linear least squares problems: for $i=1,\dots,s$,
\begin{align} \nonumber
(\delta_i^k)^2:=\ 
&\Bigl\|(\lambda_i I - hJ_0) \Phi' (\theta_0) \frac{\Delta \hat \Theta_i^k}{h} - \lambda_i \hat S_i^k\Bigr\|^2  
\\
&+ \tfrac12\eps^2 \Bigl\|  \frac{\Delta \hat \Theta_i^k}{h} + \hat\Sigma_i^k\Bigr\|_\calQ^2 + \eps^2 \Bigl\|  \frac{\Delta \hat \Theta_i^k}{h} \Bigr\|_\calQ^2 \to \min.
\label{gn-rk}
\end{align}
Note that $(\delta^k)^2=\sum_{i=1}^s(\delta_i^k)^2$.

\subsection{Solution approximation at grid points}
After $K$ iterations  we have computed the internal stages $U=(U_i)_{i=1}^s = U^K$. In order to obtain a parametric approximation $u_1=\Phi(\theta_1)$ as the starting value for the next time step, which is to approximate the non-parametric Runge--Kutta result $y_1$, there are different options:
\begin{itemize}
    \item[(i)] If $b_j=a_{sj}$ for $j=1,\dots,s$, as is the case with Radau IIA methods, then $y_1$ equals the last stage value $Y_s$. We then set $u_1=U_s$, i.e. $\theta_1=\Theta_s$ and $u_1=\Phi(\theta_1)$.
    \item[(ii)] Else, one choice is to set $\theta_1 = \theta_0 + \sum_{i=1}^s w_i (\Theta_i - \theta_0)$ and $u_1=\Phi(\theta_1)$,
    locally ignoring the nonlinearity of the parametrization map $\Phi$. This leads, however, to a severe order reduction to approximation order 1.
    \item[(iii)] A more accurate choice than (ii) is to determine $\theta_1$ from a further regularized Gau\ss--Newton iteration applied to the nonlinear least squares problem
    \begin{equation} \label{u1-Gauss}
    \bigl\| \Phi(\theta_1) -  \wt y_1 \bigr\| \to \min, \quad\text{where}\quad
    \wt y_1 = u_0 + \sum_{i=1}^s w_i (U_i - u_0).
    \end{equation}
    We use this approach for the parametric Gau\ss\ implicit Runge--Kutta methods. 
        The approximation (ii)  can provide a good starting value.
\end{itemize}

\begin{remark}
In constrast to the regularized Gau\ss-Newton iteration for the inner stages, the regularized
Gau\ss-Newton iteration for $u_1=\Phi(\theta_1)$ defined by \eqref{u1-Gauss} could be replaced by (stochastic) gradient descent.
For the inner stages $U_i=\Phi(\Theta_i)$, gradient descent is not feasible unless $hA$ is a moderately bounded operator. If $A$ is at all bounded, this would impose a severe stepsize restriction for the stiff problem. After all, to avoid such a stepsize restriction is the whole point of using implicit time integrators.
\end{remark}

\subsection{Iteration error}

The bound of Theorem~\ref{thm:gn-err} for the iteration errors extends to
implicit Runge--Kutta methods that have all eigenvalues of $\calA$ in the open right half-plane.
In particular, this includes the Radau IIA and Gau\ss\ methods of all orders, as follows e.g. from \cite[Theorem~IV.14.5]{HW91}. We compare the regularized Gau\ss--Newton iterates $U^k=(U^k_i)_{i=1}^s$ with $U^k_i=\Phi(\Theta^k_i)$ and the Newton iterates 
$Y^k=(Y^k_i)_{i=1}^s$ for the inner stages.

\begin{theorem}[Iteration error] \label{thm:gn-err-rk}
Assume that all eigenvalues of the Runge--Kutta coefficient matrix $\calA$ have positive real part.
In the setting of Section~\ref{sec:setting} and under the stepsize restriction
\begin{equation}\label{h-eps-rk}
h\delta^{k} \le c\eps^2,
\end{equation}
the iteration \eqref{gn-rk} with $J_0=A$ yields 
\begin{equation}\label{gn-err-rk}
\| U^{k+1} - Y^{k+1}\| \le  Ch (\delta^k+\rho\delta^{k-1}+\ldots +\rho^k\delta^0),
\end{equation}
where $C=\mu+\beta c$ with the bound $\beta$ of the second derivative of $\Phi$, 
and $\rho=hL_g\,\text{\rm cond}_2(\calT)$ with the Lipschitz constant $L_g$ of $g$ and $\text{\rm cond}_2(\calT)=\|\calT\|_2\,\|\calT^{-1}\|_2$. The constant $\mu$ depends only on the coefficients $a_{ij}$ of the Runge--Kutta method.
\end{theorem}

\begin{proof}
The proof extends the proof of Theorem~\ref{thm:gn-err}.  We define
\[
\hat D^k =  \Bigl( \Lambda \otimes I - h I \otimes A \Bigr) \bigl( I\otimes \Phi' (\theta_0)\bigr)  \frac{\Delta \hat \Theta^k}{h}
- ( \Lambda \otimes I) \hat S^k
\]
so that $\| \hat D^k \| \le \delta^k $. Similar to the formulas of Section~\ref{subsec:newt-irk} we get for
the difference $\Delta \hat U^k = (\calT^{-1} \otimes I) \Delta U^k$ with $\Delta U^k = U^{k+1} - U^k$
\begin{equation}\label{delthatu}
\Bigl( \Lambda \otimes I - h I \otimes A \Bigr) \frac{\Delta \hat U^k}{h} =( \Lambda \otimes I)  \hat S^k + \hat D^k
+ \Bigl( \Lambda \otimes I - h I \otimes A \Bigr) \hat W^k,
\end{equation}
where $\hat W^k = (\calT^{-1} \otimes I) W^k$ with
\[
W^k = \frac {\Phi (\Theta^{k+1}) - \Phi (\Theta^k)}{h} - \bigl( I\otimes \Phi' (\theta_0)\bigr)  \frac {\Delta \Theta^k}{h} .
\]
Since$$ 
\tfrac12\,\| \hat \Theta^{k+1} - \doubleone \otimes \theta_0 \|_{\calQ^s}^2 +
\| \Delta \hat \Theta^k \|_{\calQ^s}^2 \le  \Bigl(\frac{h\delta^k}{\eps}\Bigr)^2,
$$
we obtain
\[
h \| \hat W^k\| \le \tfrac12\beta \| \Delta \hat \Theta^k \|_{\calQ^s}^2 
+ \beta \| \Delta \hat \Theta^k \|_{\calQ^s} \, \| \hat \Theta^{k+1} - \doubleone \otimes \theta_0 \|_{\calQ^s}
\le \beta  \,\Bigl(\frac{h\delta^k}{\eps}\Bigr)^2,
\]
where $\beta$ is again a bound of the second derivative of $\Phi$.
Under the step size restriction \eqref{h-eps-rk}
we thus have 
$$
h \| \hat W^k\| \le \beta ch \delta^k.
$$
For the errors $\hat E^k = \hat U^k - \hat Y^k $ and their increments $\Delta \hat E^k = \hat E^{k+1} - \hat E^k$, taking
the difference between \eqref{delthatu} and the corresponding formula for $\Delta \hat Y^k$ yields
\begin{equation*}
\Bigl( \Lambda \otimes I - h I \otimes A \Bigr) \frac{\Delta \hat E^k}{h} =( \Lambda \otimes I)  \bigl(\hat S^k 
-  \hat R^k \bigr)+ \hat D^k
+ \Bigl( \Lambda \otimes I - h I \otimes A \Bigr) \hat W^k.
\end{equation*}
In the semilinear setting \eqref{f} we obtain, using $\hat G_Y^k = (\calT^{-1}\otimes I ) G_Y^k$
with $G_Y^k = \bigl( g(Y_i^k)\bigr)_{i=1}^s$ and $\hat G_U^k = (\calT^{-1}\otimes I ) G_U^k$
with $G_U^k = \bigl( g(U_i^k)\bigr)_{i=1}^s$, 
\[
( \Lambda \otimes I)  \bigl(\hat S^k  -  \hat R^k \bigr) = -
\Bigl( \Lambda \otimes I - h I \otimes A \Bigr)  \frac{\hat E^k}{h} + \bigl( \hat G_U^k - \hat G_Y^k \bigr) 
\]
and 
\begin{equation*}
\bigl( \Lambda \otimes I - h I \otimes A \bigr) { \hat E^{k+1}} = h  \bigl(\hat G_U^k - \hat G_Y^k \bigr) +  h \hat D^k
+ h\bigl( \Lambda \otimes I - h I \otimes A \bigr) \hat W^k .
\end{equation*}
Since the inverse of $\bigl( \Lambda \otimes I - h I \otimes A \bigr)$ is bounded by 
\begin{equation}\label{mu-rk}
\bigl\| \bigl( \Lambda \otimes I - h I \otimes A \bigr)^{-1} \bigr\| 
\le \max_{i=1,\dots,s} \max_{\Re z \le 0}|(\lambda_i -z)^{-1}| =: \mu <\infty,
\end{equation}
this yields the bound
\[
\| \hat E^{k+1} \| \le  hL_g \,\text{cond}_2(\calT)\,\| \hat E^k \| + (\mu + \beta c ) h \delta^k ,
\]
which implies, with $\rho=hL_g \,\text{cond}_2(\calT)$, 
\[
\| \hat E^{k+1} \| \le  (\mu + \beta c ) h \bigl(\delta^k + \rho \delta^{k-1}+\ldots+\rho^{k-1} \delta^0\bigr).
\]
Since $E^{k+1}=U^{k+1}-Y^{k+1}=(\calT\otimes I)\hat E^{k+1}$ and we have normalized $\|\calT\|_2=1$, we obtain \eqref{gn-err-rk}.  
\qed
\end{proof}

\section{Error bound}
\label{sec:RK-error-bound}
\subsection{Preparation: Bounds for Radau IIA and Gau\ss\ implicit Runge--Kutta methods}
Error bounds for implicit Runge--Kutta methods that are independent of the stiffness have been obtained in the literature for large classes of nonlinear stiff differential equations that include our semilinear setting of Section~\ref{sec:setting}. The concept of B-convergence, which is used for proving such error bounds, was developed by Frank, Schneid \& \"Uberhuber \cite{FSU85} and Dekker \& Verwer \cite{DV84}; see also Hairer \& Wanner \cite[Section IV.15]{HW91}.  Here we restrict our attention to Radau IIA and Gau\ss\ methods, which extend the implicit Euler method and the implicit midpoint rule, respectively, to higher order of approximation.

We recall the following local error bound; see \cite[Proposition IV.15.1]{HW91}.

\begin{lemma} [Local error]
\label{lem:loc-err-rk}
    Let $y_1$ be the result of a step \eqref{irk} of the $s$-stage Radau or Gau\ss\ implicit Runge--Kutta method starting from $y_0$.
    In the setting of Section~\ref{sec:setting}, there exists $\alpha$ such that under the condition $hL_g\le \alpha$  the numerical solution $y_1$ exists and the local error is bounded by
      \begin{equation} \label{loc-err-rk}
        \| y_1- y(t_1) \| \le C_0 \,h^{s+1} \max_{t\in[t_0,t_1]}\| y^{(s+1)}(t) \|,
    \end{equation}  
    provided the exact solution has the regularity $y\in C^{s+1}[t_0,t_1]$.
    Here, $\alpha$ and $C_0$ depend only on the particular method (Radau or Gau\ss\ with $s$ stages).
\end{lemma}

A result on the error of the modified Newton iteration in our semilinear setting 
is apparently not directly available in the literature. However, such a result is readily obtained by extending the argument in the proof part (i) of Lemma~\ref{lem:ie-newton}.

\begin{lemma} [Newton iteration]
\label{lem:newton-rk}
    Let $Y=Y_i$ ($i=1,\dots,s)$ be the stages of a step \eqref{irk} of the $s$-stage Radau or Gau\ss\ implicit Runge--Kutta method starting from $y_0$, and let $Y^k=(Y_i^k)$ be the stages after $k\ge 0$ modified Newton iterations \eqref{newton} with $J_0=A$. Then, there exist $\alpha>0$ and $\kappa>0$ such that for step sizes bounded by $hL_g\le\alpha$, 
    \begin{align} \label{newton-err-rk}
        \| Y^{k+1} - Y \| &\le hL_g \kappa \,\| Y^{k} - Y \| .
    \end{align}
    Here, $\alpha$ and $\kappa$ depend only on the particular method (Radau or Gau\ss\ with $s$ stages). The norm in \eqref{newton-err-rk} is the norm in $\calH^s$.
\end{lemma}

\begin{proof} We rewrite \eqref{newton-rk} as
$$
(\lambda_i - hA) \hat Y_i^{k+1} = \lambda_i y_0 + h \hat G_i^k
$$
where $\hat G_i^k$ is the $i$th component of $\hat G^k= (\calT^{-1}\otimes I)G^k$ with
$G^k=\bigl( g(Y_i^k) \bigr)_{i=1}^s$. Similarly, the transformed stage vector $\hat Y_i$ satisfies the equation
$$
(\lambda_i - hA) \hat Y_i = \lambda_i y_0 + h \hat G_i
$$
where $\hat G_i$ is defined in the same way as $\hat G_i^k$ but with $Y_i$ instead of $Y_i^k$ in the argument of $g$. Subtracting the two equations, multiplying with  $(\lambda_i - hA)^{-1}$ and taking norms yields
$$
\| Y^{k+1} - Y \| \le \| \calT \|_2 \,
\| \hat Y^{k+1} - \hat Y \| \le \mu \,  \text{cond}_2(\calT) \, hL_g \, \| Y^{k} - Y \|
$$
with $\mu$ of \eqref{mu-rk}. 
\qed
\end{proof}

The following lemma is closely related to $B$-stability (a concept due to Butcher \cite{But75}). It was first proved in \cite{DV84}; see \cite[Proposition IV.15.2]{HW91}.

\begin{lemma}[Stable error propagation]
 \label{lem:stable-rk}
    Let $y_{1}$ and $\wt y_{1}$ be the results of a step of the $s$-stage Radau or Gau\ss\ implicit Runge--Kutta method starting from $y_0$ and $\wt y_0$, respectively. Then, there exist $\alpha>0$ and $\gamma>0$ such that for step sizes bounded by $hL_g\le\alpha$, 
    \begin{equation} \label{stable-rk}
    \| y_{1} - \wt y_{1} \| \le (1+\gamma hL_g) \, \| y_{0} - \wt y_{0} \| ,
    \end{equation}
    where $\alpha$ and $\gamma$ depend only on the particular method (Radau or Gau\ss\ with $s$ stages).
\end{lemma}

Combining Lemmas~\ref{lem:loc-err-rk} and~\ref{lem:stable-rk} via Lady Windermere's fan (as in Figure~\ref{fig:Windermere-yn-y(tn)}) yields the global error bound, for $0\le nh \le T=\bar t - t_0$,
\begin{equation}\label{glob-err-rk}
\| y_n - y(t_n) \| \le Ch^s\,\frac{e^{\gamma L_g T} -1}{\gamma L_g}\, \max_{t\in[t_0,\bar t]}\| y^{(s+1)}(t) \|.
\end{equation}

\begin{remark}[Order reduction phenomenon]
The order $s$ in this bound is smaller than the classical orders $2s-1$ and $2s$ for Radau IIA and Gau\ss\ methods, respectively, which are attained when these methods are applied to smooth nonstiff differential equations. This order reduction phenomenon for stiff problems is well known; see e.g. \cite[Section IV.15]{HW91}. The order $s$ in \eqref{glob-err-rk} appears to be optimal over the whole class of differential equations considered in Section~\ref{subsec:ivp}. However, higher order may be attained for subclasses of such problems. For example, for approximating smooth solutions of partial differential equations \eqref{ivp}--\eqref{A} with periodic boundary conditions or on the full space $\R^d$,  the full classical order is attained. For parabolic problems on smooth bounded domains with Dirichlet or Neumann boundary conditions, non-integer orders larger than $s+1$ are attained.
\end{remark}

\subsection{Errors}
We extend the error bound of Theorem~\ref{thm:ie-err} to Radau IIA and Gau\ss\ implicit Runge--Kutta methods of arbitrary order. Proceeding as in Section~\ref{subsec:ge-ie}, we define
\begin{align*}
\delta_n &:= \delta_n^{K}+\rho\delta_n^{K-1}+\ldots +\rho^K\delta_n^0
\\
\eta_n &:= h^{K-1}\,(\kappa L_g)^{K-1}\, \| U_n^0 - \wt Y_n \|,
\end{align*}
where $\wt Y_n$ is the vector of internal stages of the Runge--Kutta method starting from $u_n$, while $U_n^0=\wt Y_n^0$ is the starting vector for the both the regularized Gau\ss\--Newton iteration for computing $U_n^K$ and the Newton iteration that converges to $Y_n$.

\begin{theorem} [Global error bound]  \label{thm:rk-err}
In the setting of Section~\ref{sec:setting} and under the conditions $hL_g\le\alpha$ and
\begin{equation}\label{h-eps-global-rk}
h (\delta_n^{K-1}+\ldots+\delta_n^0) \le c\eps^2, \qquad 0\le nh \le T =\bar t - t_0,
\end{equation} 
we have with $\delta=\max\delta_n$ and $\eta=\max \eta_n$
\begin{equation}
    \| u_n - y_n \| \le  C_1\delta + C_2\eta, \qquad 0\le nh \le T ,   
\end{equation}
where $C_1= C\bigl(e^{\gamma L_g T} -1\bigr)/(\gamma L_g)$ and $C_2 = e^{\gamma L_g T} -1$,  with $C$ of Theorem~\ref{thm:gn-err-rk}.
\\
Combined with the error bound \eqref{glob-err-ie} of the non-parametric implicit Euler method, this yields a global error bound (for $K\ge s+1$ iterations per time step)
$$
\|u_n-y(t_n)\|=O(h^s+\delta).
$$
\end{theorem} 

This result is proved by combining Theorem~\ref{thm:gn-err-rk} and the lemmas of the previous subsection via Lady Windermere's fan (as in Figure~\ref{fig:Windermere-un-yn}) in the same way as in proving Theorem~\ref{thm:ie-err}.

\section{Implementation and numerical experiments}
\label{sec:numerics}
We briefly describe some elements of the implementation of the presented methods. We
present results for evolving neural network parametrizations of the one-dimensional advection equation and heat equation,
which are deliberately chosen as simple test equations that show nontrivial behaviour
of the numerical methods.

\subsection{On the computation of integrals}
The main challenge in realizing the proposed integrators in a setting of $\calH=L^2(\Omega)$ lies in the numerical quadrature of the integrals resulting from solving the least-squares problem \eqref{gn-1} in the Gau\ss--Newton iteration by the normal equations. This  requires the assembly of integrals
\begin{align}\label{eq:B-matrices}
	\int_{\Omega}  B(\theta)^*B(\theta) \, \mathrm d x, 
	\quad \quad \int_{\Omega}  B(\theta)^* r(\theta) \, \mathrm d x,
\end{align}
where $\theta\in\mathcal Q$  is given and the form of $B(\theta)\in L^2(\Omega)$ is determined by the numerical integrator. The implicit Euler method requires $B(\theta) = (I -hA) \Phi'(\theta)$, the implicit midpoint rule requires the evaluation of the integral with $B(\theta)=(I -\frac{h}{2}A) \Phi'(\theta)$. The term $r(\theta)$ corresponds to some right-hand side derived from the partial differential equation in combination with the time integration scheme. 

These integrals are discretized by a standard quadrature. In our numerical experiments we use a composite Gaussian quadrature in one dimension, with $4$ nodes on each subinterval. The numerical quadrature of the first integral is, by far, the computationally most expensive. Here, we require, for $M$ spatial quadrature points, $\mathcal O(M Q^2)$ operations, where $Q$ is the dimension of the parameter space $\mathcal Q = \mathbb R^Q$. For the evaluations of the matrices $B(\theta)(x)$, we use the automatic differentiation framework Tensorflow. In order to mitigate effects of finite precision arithmetics, we use double precision for all experiments. Studying the effects of single precision arithmetic and errors introduced by the numerical integration is beyond the scope of the present paper.

\subsection{On initial values}
Throughout the previous sections, we assumed an initial approximation $u_0= \Phi(\theta_0)\approx y_0 \in \calH$ to be given. In order to construct such usable initial approximations, we proceed as described in our earlier work \cite[Section~9.1]{FLLN24} and approximate the initial value problem (with a fictitious time $\tau$ and with constant right-hand side)
\begin{align}\label{eq:fitting-ode}
	\frac{dz}{d\tau}(\tau) &= y_0 - \Phi(\tilde{\theta}_{0}), \qquad 0\le \tau \le 1,
    \\
    z(0)&= \Phi(\tilde{\theta}_{0}),
    \nonumber
\end{align}
with the exact solution $z(1)=y_0$ at $\tau=1$.
Here, $\tilde{\theta}_0$ has been constructed by roughly fitting $\Phi$ to the initial data by a standard (typically gradient-based) nonlinear optimizer. In our experiments, we used the standard Adams optimizer. For the time discretization, we choose the regularized parametric integrator based on the classical explicit fourth-oder Runge--Kutta method (see \cite{FLLN24}) with $100$ time steps, and we use $\varepsilon=10^{-4}$ as the regularization parameter. This process is then iterated once, that is, the resulting approximation is again used to formulate the initial value problem \eqref{eq:fitting-ode}. Figure~\ref{fig:initfuns} shows the effectiveness of this approach and plots two neural networks that were fitted to a Gaussian and a piecewise linear hat function. The errors of these approximations are visualized in Figure~\ref{fig:initfuns-errors}.

These two trained neural networks are used as the initial parametrization in all our experiments. They take the role of a smooth initial condition and a non-smooth initial condition, respectively. They are described by the expressions
\begin{align*}
	y_0(x) = e^{-4x^2}, 
	\quad \quad 
	y_0(x) = 
	\begin{cases}
		1-|x| & \quad \text{for } |x|\le \frac{1}{2} ,\\
		0 &\quad \text{else} .
	\end{cases}
\end{align*}
The functions are approximated on the interval $[-\pi,\pi]$. As the fixed neural network architecture, we choose four hidden layers with five neurons each, which uses $131$ parameters. As the pointwise activation functions, we choose the smooth function $\tanh$. The periodic boundary conditions are embedded into the network architecture, by setting the input layer $\phi_{\mathrm{in}} \, \colon \mathbb R \rightarrow \mathbb R^5$ to 
\begin{align*}
	(\phi_{\mathrm{in}}(x))_{i} &= \sin(x+b_i), 
\end{align*}
where $b\in\mathbb R^5$ is the first bias vector, which is part of the overall vector of parameters $\theta\in\R^{131}$. The hidden layers $  \phi_j \colon \mathbb R^5 \rightarrow \mathbb R^5 $ (for $j=1,2,3,4$) are each determined by a weight matrix $W_j\in \mathbb R^{5\times 5}$ and a bias vector $b_j\in \mathbb R^5$ and read, including the nonlinear activation function
\begin{align*}
         \phi_j(y) & = \tanh(W_jy+b_j) \quad \text{for}\quad j=1,\dots,4.
\end{align*}
Finally, the output layer is an affine map $\phi_{\mathrm{out}}\colon \mathbb R^5\rightarrow \mathbb R$, which is determined by a weight vector $w_{\mathrm{out}} \in \mathbb R^5$ and a scalar bias $b_{\mathrm{out}}\in \mathbb R$ via
\begin{align*}
    \phi_{\mathrm{out}}(z) = w_{\mathrm{out}}^Tz +b_{\mathrm{out}}.
\end{align*}
All weights and biases are collected in the parameter vector $\theta \in \mathcal Q=\mathbb R^{131}$. The parametrization $\Phi(\theta)$, evaluated at a position $x\in [-\pi,\pi]$, then reads
\begin{alignat*}{5}
    \Phi(\theta)[x] = \phi_{\mathrm{out}}\,&\circ \,\phi_4\,&&\circ \,\phi_3\,&&\circ \,\phi_2\,&&\circ \,\phi_1\,&&\circ \,\phi_{\mathrm{in}}(x), \\
    \text{with } \quad 6 \ \ &+  30 &&+  30 &&+  30 &&+ 30 &&+  5 \quad = 131 \quad \text{parameters}.
\end{alignat*}
Despite the low regularity of the functions, the neural network is able to approximate these functions well (see Figure~\ref{fig:initfuns-errors}). Surprisingly, the absolute values of the weights remain very moderate, as is seen in Figure~\ref{fig:weights}, despite the strong curvature exhibited near the points of low regularity of the initial data~$y_0$.

\subsection{On the choice of the regularization parameter $\varepsilon$}\label{sect-reg-control}

We shortly describe how the regularization parameter $\varepsilon$ has been chosen for our experiments, and how we motivate this choice. For sufficiently large $\varepsilon$, we observe that $\delta = \mathcal O(\varepsilon)$, which is numerically observed in Figure~\ref{fig:eps-to-delta} (and, for the linearly implicit Euler method, seen from a standard linear algebra argument). We therefore start with a relatively large $\varepsilon$ and then continually decrease $\varepsilon$ until certain criteria are violated, in order to find a good approximation of the largest $\varepsilon$ that fulfills predetermined conditions that are informed by the error analysis.

Consider some user-defined tolerance  $\delta_\text{tol}$. Initially, we set a very rough regularization $\varepsilon=1$. Then, we set $\varepsilon\leftarrow \frac{\varepsilon}{2}$ and compute $\delta(\varepsilon)$, until either
\begin{equation}
	\delta(\varepsilon)<\delta_{\text{tol}}\quad \text{or}\quad  \frac{3}{2}\delta_{\text{min}} < \delta(\varepsilon)
	\quad \text{or} \quad \frac{\delta(\varepsilon)}{\varepsilon} > 10.
\end{equation}
The number $10$ is an arbitrarily specified constant, which allows for step sizes that break the step size condition of Lemma~\ref{thm:gn-err}, but for which the method is still expected to be stable. The parameter $\delta_{\text{min}}$ denotes the smallest defect encountered through the search. After the method terminates, we use the $\varepsilon$ that corresponds to this smallest defect $\delta_{\text{min}}$ and start with the time integration process. 

The initial value of $\varepsilon$ is updated in every time step, in order to heuristically enforce these same restrictions. If the crucial factor $\delta(\varepsilon)/\varepsilon$, which determines the maximal rate of change of the parameters (see Remark~\ref{rem:deviation-theta}) becomes larger than $100$, or the defect becomes significantly smaller than the tolerance $\delta<\delta_{\text{tol}}/{10}$, we double the regularization parameter to $\varepsilon\leftarrow 2 \varepsilon$. When the defect becomes significantly larger than the tolerance, i.e. $\delta(\varepsilon)>10\,\delta_{\text{tol}}$, we decrease the regularization by setting $\varepsilon\leftarrow \frac{\varepsilon}{2}$, as long as the quotient $\delta/\varepsilon$ is smaller than a specified constant, which in our experiments was set to $10$.
Throughout our experiments, we set 
$$ \delta_{\text{tol}} = h^k,$$
where $k$ was the expected order of convergence of the underlying time integration scheme.

\subsection{Other algorithmic parameters} Throughout the experiments, we used a composite Gauß quadrature, which partitions $[-\pi,\pi]$ into $20$ subintervals of equal size and approximates the integral on each partition by $4$ nodes. During each time step, we applied $20$ Gauß--Newton iterations per time step. When applied to the low-regularity setting, we introduced a slight damping into the Gauß--Newton method, with a damping factor of $\lambda=0.9$ and used a finer quadrature, partitioning $[-\pi,\pi]$ into $50$ equidistant subintervals instead of $20$.

\begin{figure}
	\includegraphics[width=1.\textwidth,trim = 0mm 0mm 0mm 0mm]{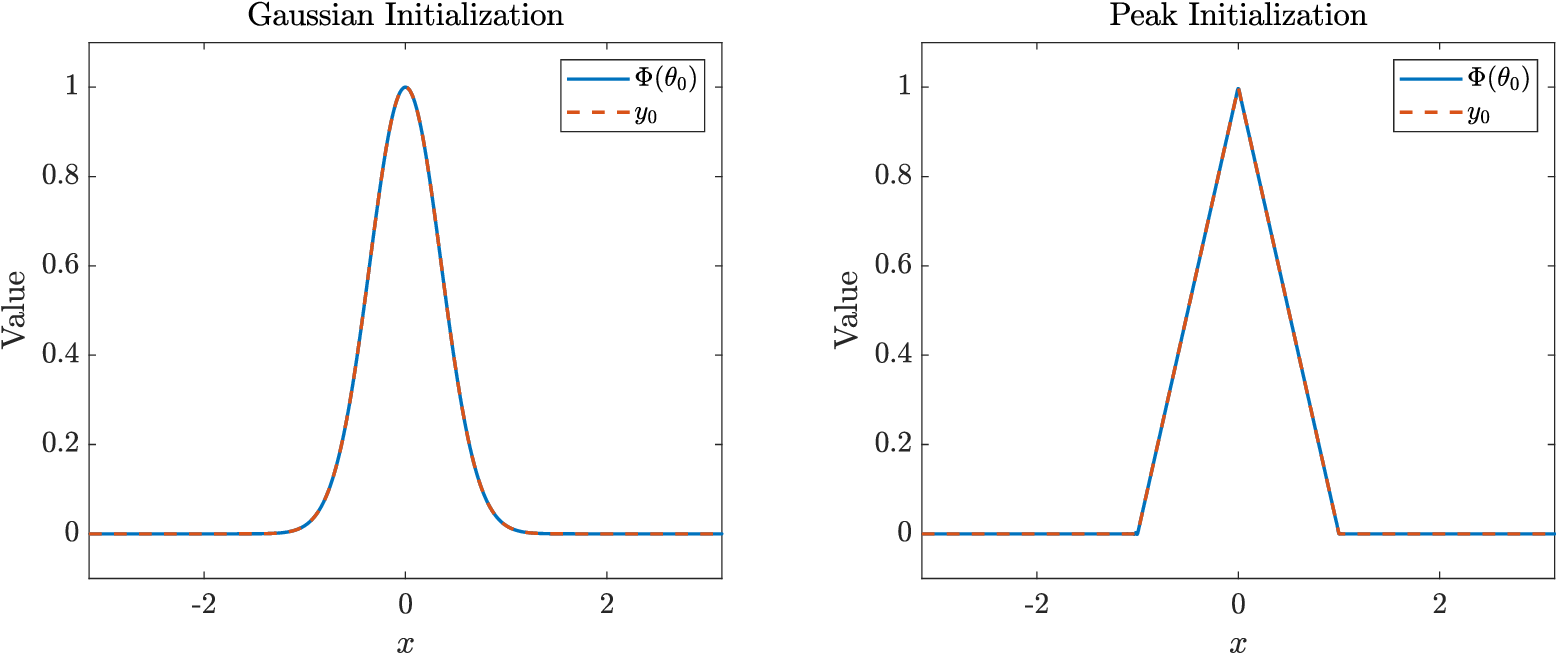}
	\caption{Different initial functions, which are used for the numerical experiments of the transport equation and the heat equation in one dimension. }
	\label{fig:initfuns}
\end{figure}
\begin{figure}
	\includegraphics[width=1.\textwidth,trim = 0mm 0mm 0mm 0mm]{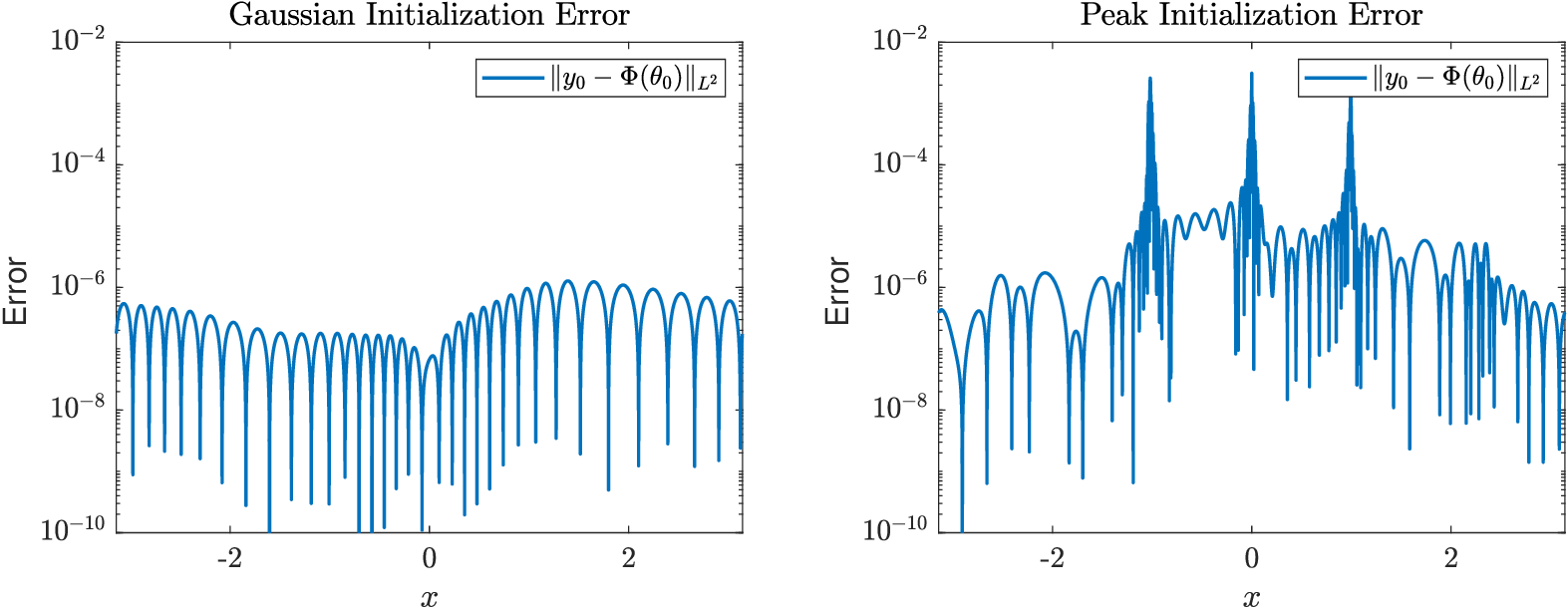}
	\caption{ Errors of the initial neural network approximations shown in Figure~\ref{fig:initfuns}. Most of the error is localized around points of low regularity.}
	\label{fig:initfuns-errors}
\end{figure}
\begin{figure}[h!]	\includegraphics[width=1.\textwidth,trim = 0mm 0mm 0mm 0mm]{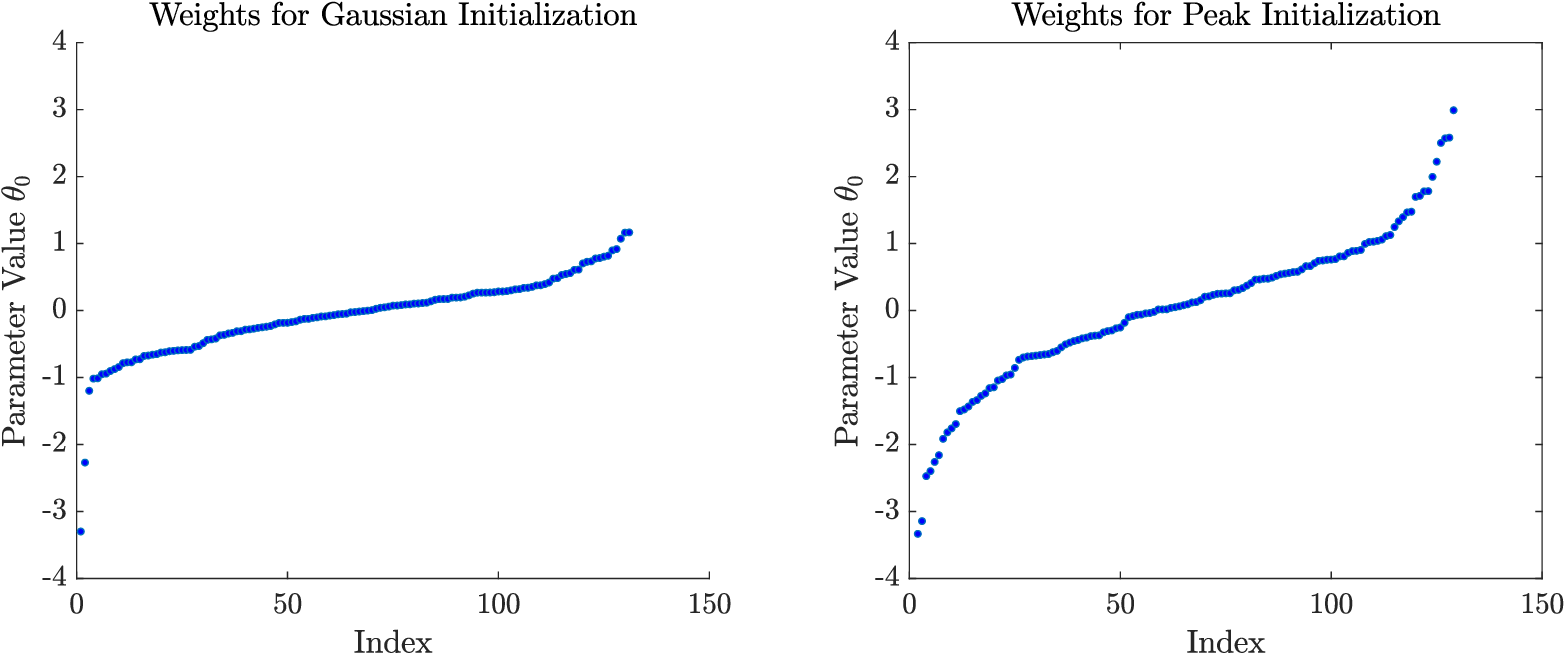}
	\caption{Sorted weights of the corresponding neural network approximations. The absolute value of the parameters is bounded by a very moderate bound, despite the shallow network architecture and the high curvature at several points of the approximation.}
	\label{fig:weights}
\end{figure}
\begin{figure}[h!]
	\includegraphics[width=1.\textwidth,trim = 0mm 0mm 0mm 0mm]{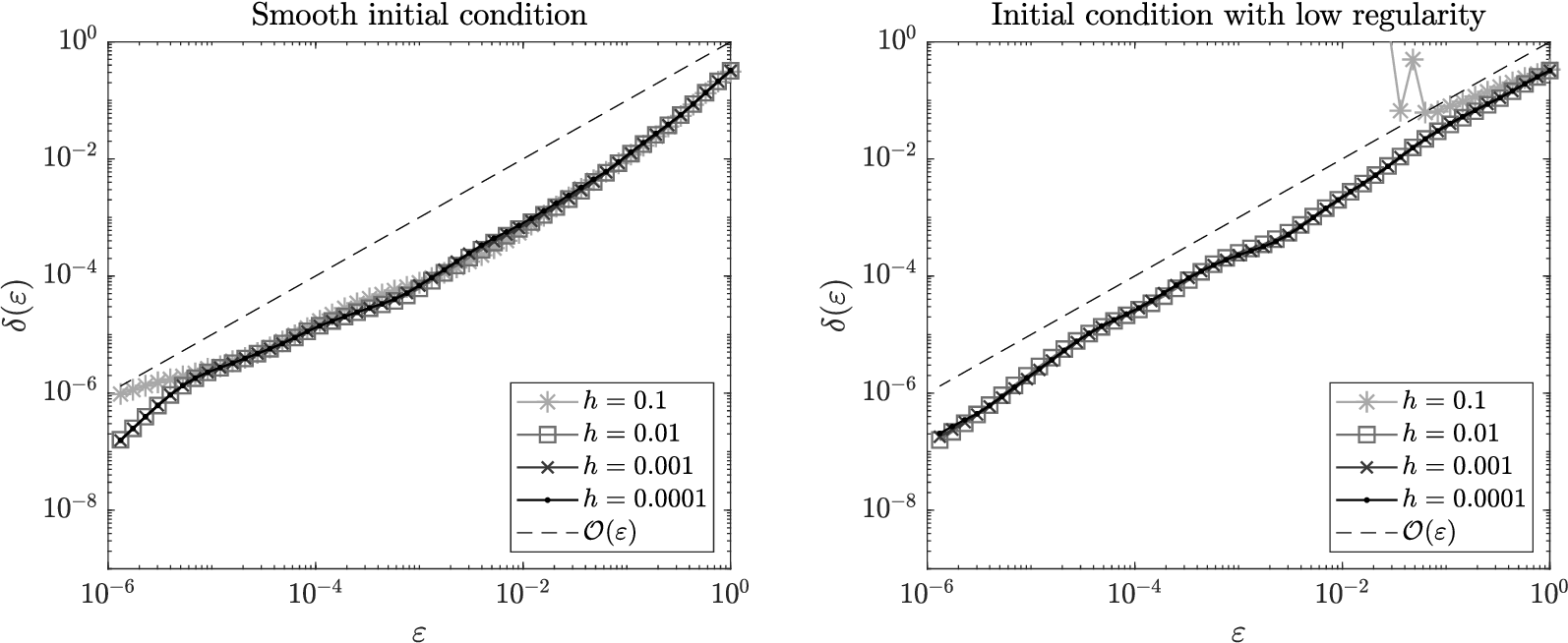}
	\caption{ For the implicit midpoint rule, we visualize the final defect $\delta^k$ at the initial time $t_0$, in dependence of the regularization parameter $\varepsilon$. These plots were generated for the implicit midpoint rule applied to the transport equation, with the initial conditions shown in Figure~\ref{fig:initfuns}.}
	\label{fig:eps-to-delta}
\end{figure}

\begin{figure}[h!]
	\includegraphics[width=1.\textwidth,trim = 0mm 5mm 0mm 0mm]{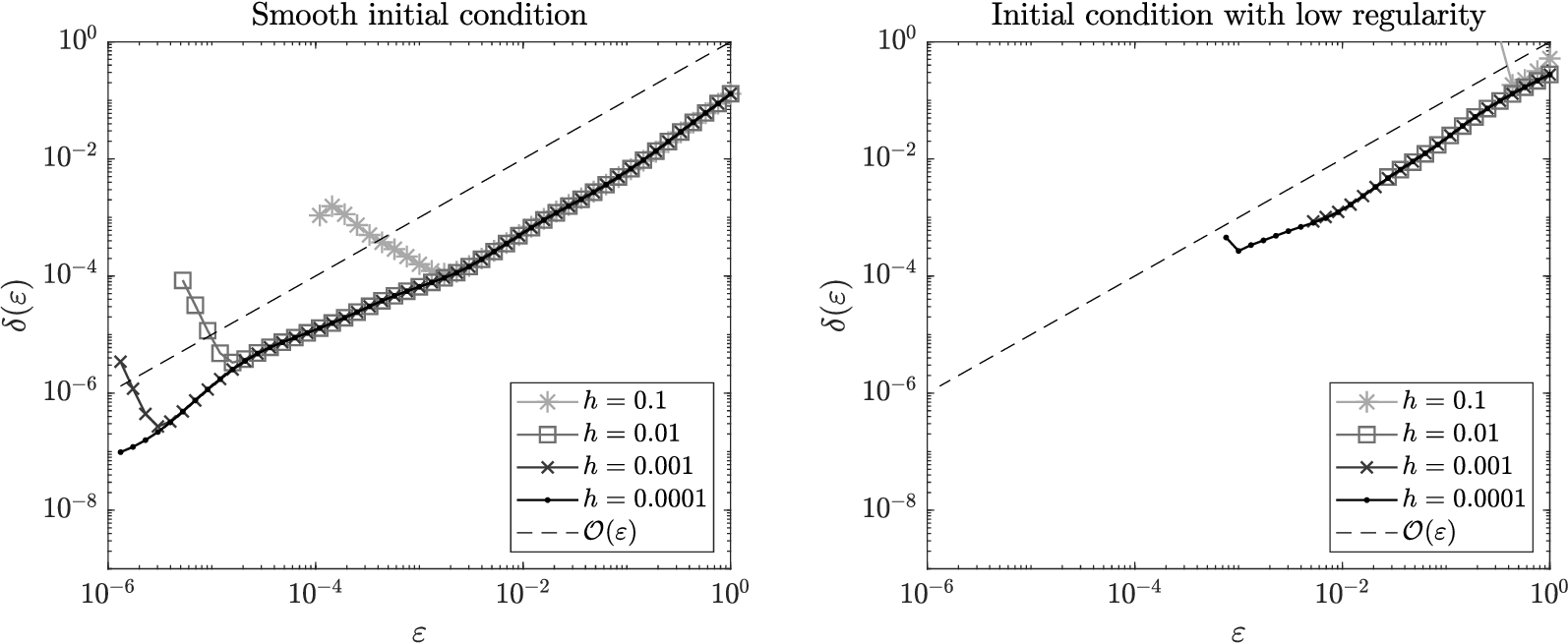}
	\caption{The same plots for the $2$-stage Gauß Runge--Kutta method. We observe that the corresponding (simplified) Gauß--Newton iteration is substantially affected by small singular values, which are present for $\varepsilon\rightarrow 0$. The effect is significantly stronger when the neural network approximates the hat function, which has only $C^0$-regularity.}
	\label{fig:eps-to-delta-RK}
\end{figure}

\subsection{Case study: The transport equation}
Consider the transport equation in one dimension: 
\begin{align}\label{eq:transport}
	\partial_t{y} - \partial_x y = 0, \quad \text{in } \quad [-\pi,\pi],
\end{align}
completed by periodic boundary conditions. This is of the form \eqref{ivp}--\eqref{A} with $A=\partial_x$ and $g=0$, considered on the Hilbert space $L^2(-\pi,\pi)$. The regularized implicit midpoint rule is determined by the linear minimization problem, where the matrix and the right-hand side are determined by \eqref{eq:B-matrices}, with 
$$
B(\theta) = (I-\tfrac{h}{2}\partial_x)\Phi(\theta).
$$
We note that the effort necessary for the computation of $\partial_x\Phi'(\theta_n) = \left(\partial \Phi(\theta_n) \right)' $
is, for moderately sized networks, roughly of the same complexity as the Jacobian $\Phi'(\theta_n)$ and can effectively be realized with automatic differentiation.
The updates of the Gauß--Newton method can be computed directly, since
\begin{align} \nonumber
\bigl(\delta^k\bigr)^2 := \ 
&\| \bigl(I -hA\bigr) \Phi'(\theta_0) \Delta \theta_1^k/h + r^k \|^2 \\
&+ \tfrac12\eps^2 \| \Delta \theta_1^k/h + \sigma^k \|_\calQ^2 
+ \eps^2 \| \Delta \theta_1^k/h \|_\calQ^2 
\to \min.
\label{gn-1-num}
\end{align}
In Figure~\ref{fig:transport-errors}, we visualize the errors for the implicit Euler method and the implicit midpoint rule, with each of the initial values shown in Figure~\ref{fig:initfuns}. As the error measure, we approximate the $L^2$-norm at the final time $T=1$. 

As expected from classical results, we have second-order convergence for the implicit midpoint rule for smooth initial data, but an order recuction to order 1 for the nonsmooth initial data. The implicit Euler method shows first-order convergence for both initial conditions.

\subsection{Case study: The heat equation}
We further consider the heat equation in one dimension: 
\begin{align}\label{eq:heat}
	\partial_t{y} - \partial_{xx} y = 0, \quad \text{in } \quad [-\pi,\pi],
\end{align}
again completed by periodic boundary conditions.
This is of the form \eqref{ivp}--\eqref{A} with $A=\partial_{xx}$ and $g=0$, again considered on the Hilbert space $L^2(-\pi,\pi)$.
As before, we note that the computation of $\Delta \Phi'(\theta_n) = \left(\Delta \Phi(\theta_n) \right)' $
is, for moderately deep networks, roughly of the same complexity as the Jacobian $\Phi'(\theta_n)$ and can again effectively be realized with automatic differentiation frameworks.
A reference approximation is computed by using a spatial Fourier base and integrating exactly in time. The error is measured at the final time $T=1$. 
In Figure~\ref{fig:heat-errors}, we visualize the errors for the implicit Euler method and the implicit midpoint rule applied to the heat equation, again with each of the initial values shown in Figure~\ref{fig:initfuns}. Here, we report reasonably good convergence results for both the implicit Euler method and the implicit midpoint rule, with surprisingly small error constants. Here, the implicit midpoint rule required more Gauß--Newton iterations per timestep ($50$) and the Jacobian was recomputed during each iteration. These restrictive parameter choices were made for the sake of the convergence plot, a cheaper method with "messier" convergence behavior can be obtained by relaxing these choices.

The approach does not produce reliable results for the heat equation when the initial values are piecewise linear (and are no longer smooth). Here, we simulate until the final time $T=0.1$ and do not report any convergence of the proposed integrators. We conjecture that the presented methodology works well for strong solutions of partial differential equations -- for the treatment of weak solutions with low regularity, new insights are necessary. A possible remedy would be to use a $\mathcal H = H^{-1}([\pi,\pi])$-setting, though then the implementation of the least squares problem \eqref{par-ie-1} becomes challenging.

\begin{figure}[h!]
	\includegraphics[width=1.\textwidth,trim = 0mm 0mm 0mm 0mm]{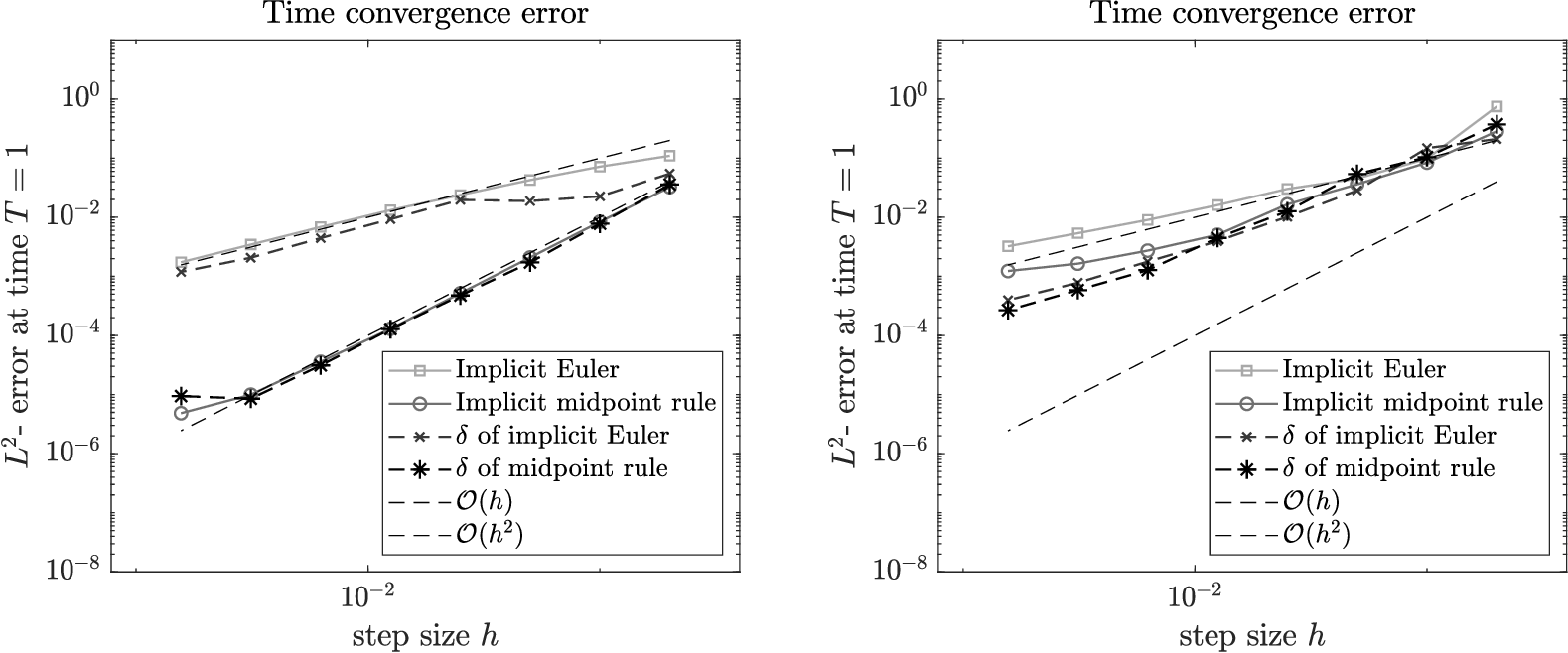}
	\caption{For the two initial conditions in Figure~\ref{fig:initfuns}, we approximate the solution of the transport equation. For less spatial regularity of the initial value, we obtain less temporal convergence.}
	\label{fig:transport-errors}
\end{figure}
\begin{figure}[h!]
	\includegraphics[width=1.\textwidth,trim = 0mm 0mm 0mm 0mm]{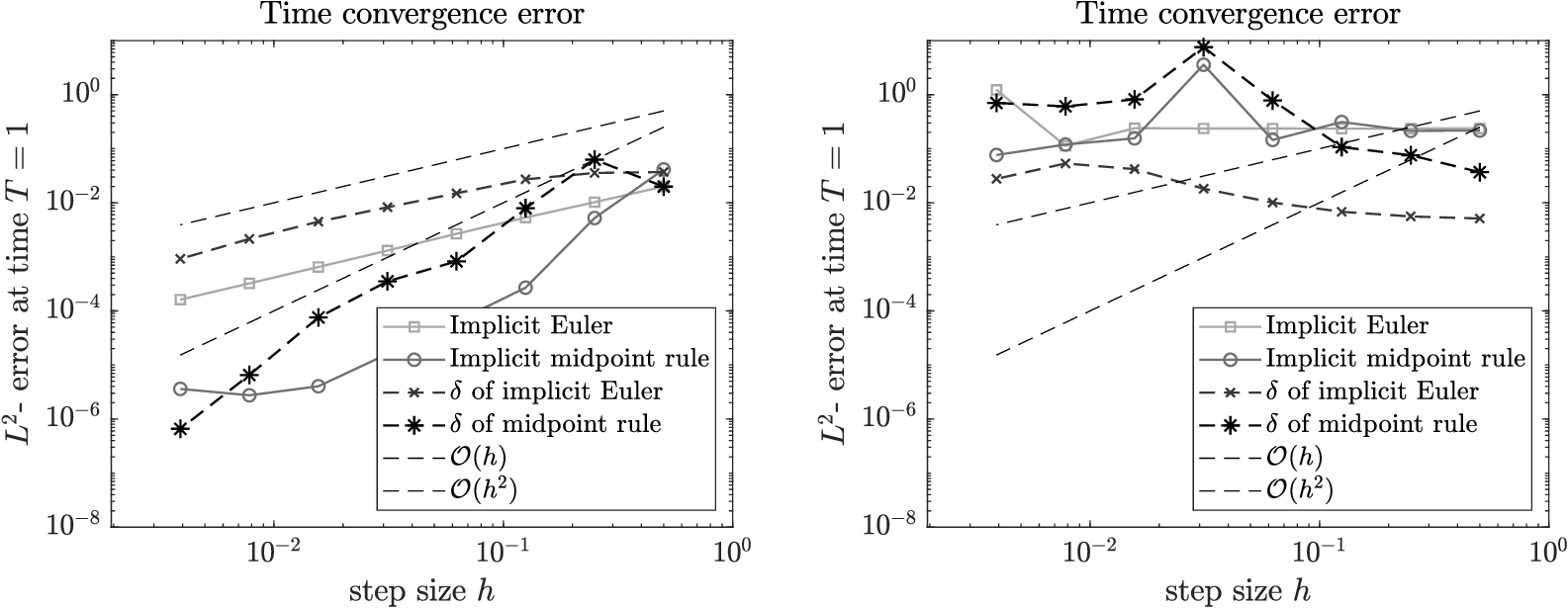}
	\caption{For the two initial conditions in Figure~\ref{fig:initfuns}, we approximate the solution of the heat equation. For less spatial regularity of the initial value, we obtain less temporal convergence.}
	\label{fig:heat-errors}
\end{figure}

\subsection{Experiments with multistage methods}
In the following, we investigate the performance of multistage methods. We show convergence plots with varying, but fixed regularization parameters $\varepsilon$ and observe the behaviour for $h\rightarrow 0$. We approximate the transport equation with the $2$-stage Gauß Runge--Kutta method, determined by 
\begin{align*}
\calA = 
\begin{pmatrix}
	\tfrac{1}{4} & \tfrac{1}{4} - \tfrac{\sqrt{3}}{6}\\[1mm]
	\tfrac{1}{4} + \tfrac{\sqrt{3}}{6} & \tfrac{1}{4}
\end{pmatrix},
\quad b = \begin{pmatrix}
	\tfrac{1}{2} \\[1mm] \tfrac{1}{2}
\end{pmatrix}
\quad \text{and}
\quad c = \begin{pmatrix}
\tfrac{1}{2}- \tfrac{\sqrt{3}}{6}
\\[1mm]
\tfrac{1}{2}+\tfrac{\sqrt{3}}{6}
\end{pmatrix}.
\end{align*}
As a second multistage method, we consider the $2$-stage Radau IIA method, which is L-stable (see \cite[Proposition 3.8]{HW91}) and therefore particularly suited for parabolic problems. It is determined by the coefficients

\begin{align*}
\calA = 
\begin{pmatrix}
	\tfrac{5}{12} & -\tfrac{1}{12} \\[1mm]
	\tfrac{3}{4}  & \tfrac{1}{4}
\end{pmatrix},
\quad b = \begin{pmatrix}
	\tfrac{3}{4} \\[1mm] \tfrac{1}{4}
\end{pmatrix}
\quad \text{and}
\quad c = \begin{pmatrix}
\tfrac{1}{3}
\\[1mm]
1
\end{pmatrix}.
\end{align*}
We now apply both of these multistage methods to the transport equation, with the $\varepsilon$-control proposed in Section~\ref{sect-reg-control}. In each Gauß--Newton iteration, we solve the decoupled linearized minimization problems \eqref{gn-rk}. We obtain the results shown in Figures \ref{fig:transport-errors-RK} and \ref{fig:heat-errors-RK}.

\begin{figure}[h!]
	\centering
	\includegraphics[width=.7\textwidth,trim = 0mm 0mm 0mm 0mm]{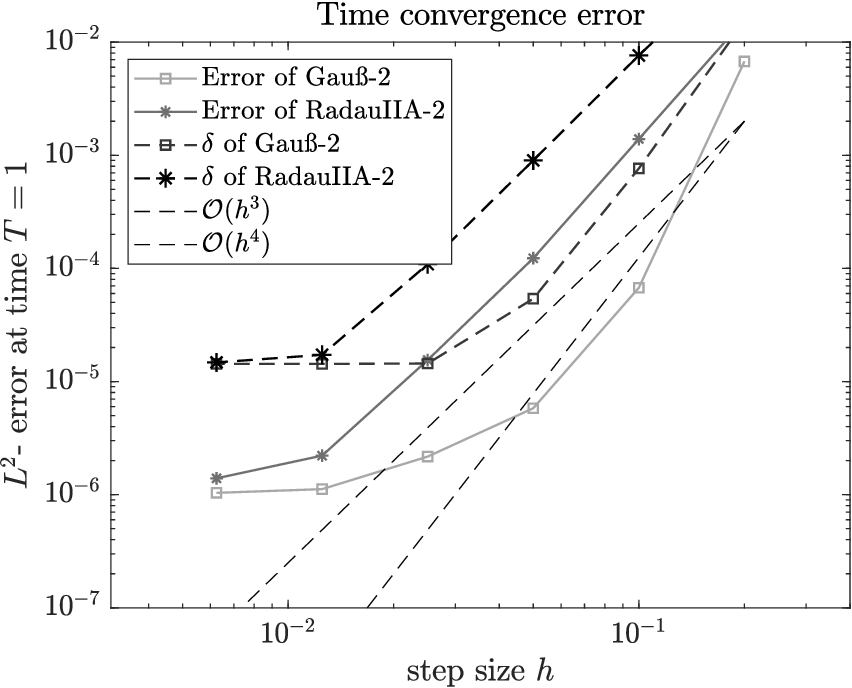}
	\caption{Approximation with $2$-stage implicit Runge--Kutta methods. The transport equation with smooth initial condition is approximated by the adaptively regularized parametric method. We observe third-order convergence of the solver based on the $2$-stage Radau IIA method, which only flattens out for small stepsizes. For the $2$-stage Gau\ss \ method we observe fourth-order convergence down to an error below $10^{-5}$.}
	\label{fig:transport-errors-RK}
\end{figure}

\begin{figure}[h!]
	\centering
	\includegraphics[width=.7\textwidth,trim = 0mm 0mm 0mm 0mm]{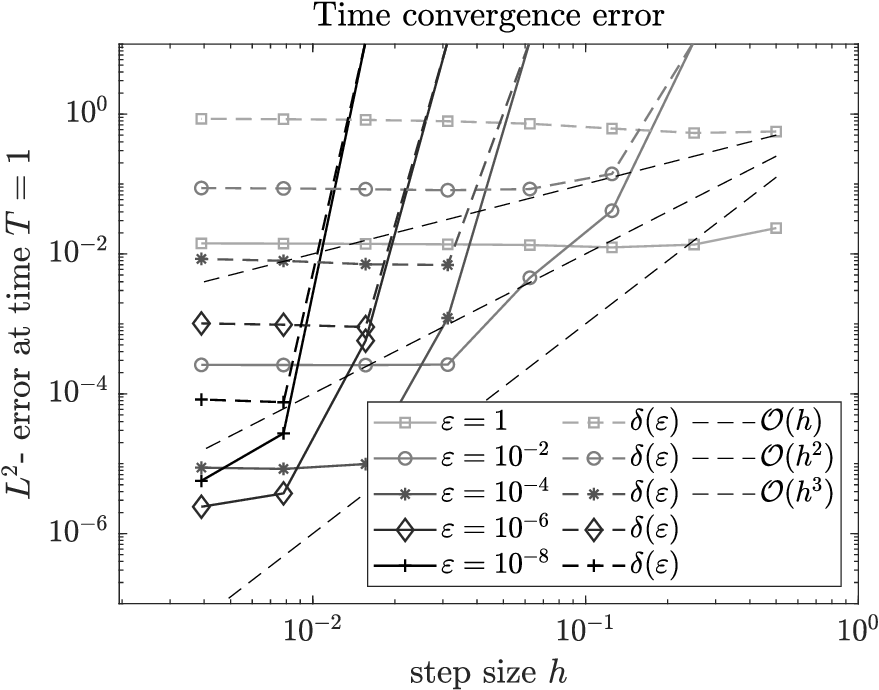}
	\caption{Error vs.~stepsize of the scheme approximating solutions to the heat equation with smooth initial conditions with the $2$-stage Radau method and various $\varepsilon$.}
	\label{fig:heat-errors-RK}
\end{figure}

\subsection*{Decaying defects of the Gauß--Newton iteration}
The theoretical analysis presented in this paper depends crucially on the smallness of the defect size $\delta^k$, as it appears e.g. in \eqref{gn-1}. We expect these defects to decay with the number of regularized Gauß--Newton iterations. An efficient error estimate would then be provided by Theorem~\ref{thm:gn-err-rk}, which only depends, up to higher-order error terms, on the defect size of the final Gauß--Newton iteration. 
As our test problem, we consider both the smooth initial condition and the piecewise linear hat function, as shown in Figure~\ref{fig:initfuns}, and propagate the fitted networks along the transport equation, until the small final time $T=0.01$, with $N=20$ time steps of the $2$-stage Gauß Runge--Kutta method. This simulation is conducted with various fixed values of $\varepsilon$.
Figure~\ref{fig:defects} visualizes the defects, as they appear in the Gauß--Newton iterations.
\begin{figure}
    \centering
    \includegraphics[width=1\linewidth]{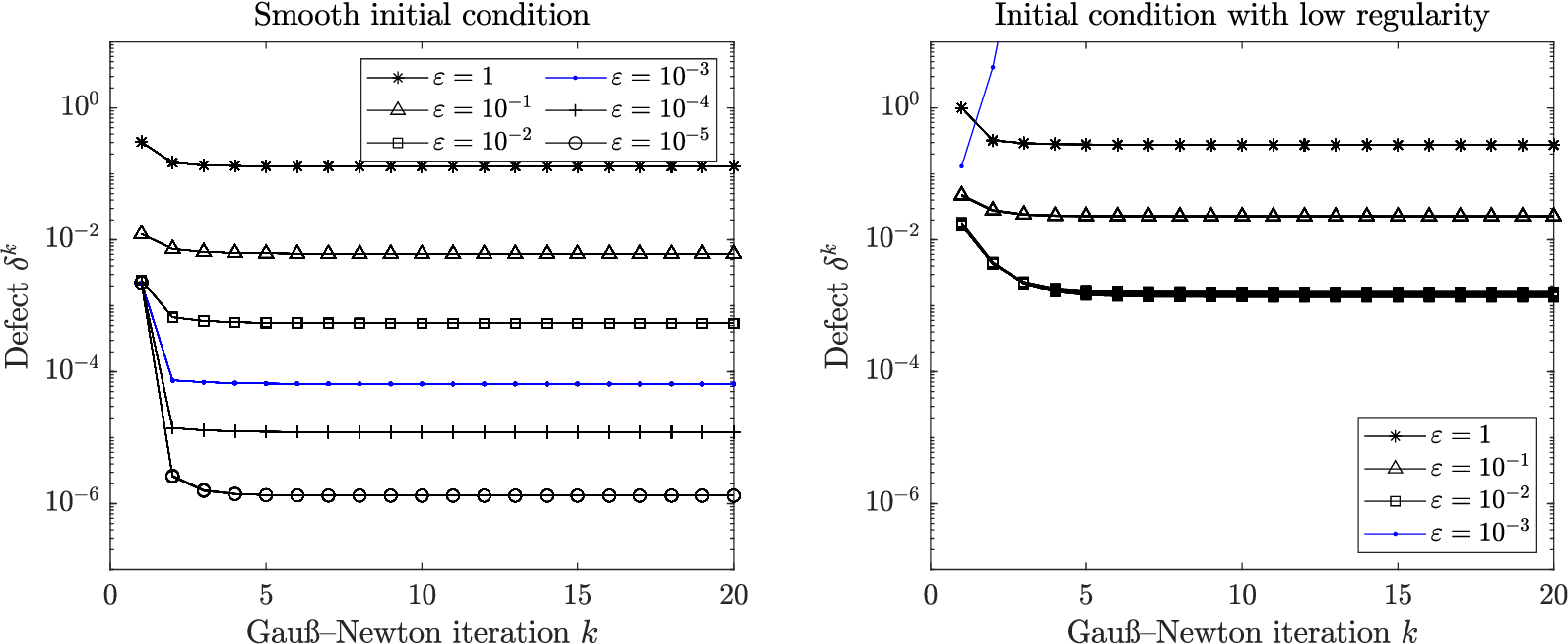}
    \caption{Defect sizes $\delta^k$, as they occur during the time integration of the transport equation, over the small time interval $[0,0.01]$, for various values of the regularization $\varepsilon$, with $N=20$ time steps of the $2$-stage Gauß Runge--Kutta method. Each line corresponds to the defects that occur during a time step, most lines lie precisely on top of each other (in the left plot $120$ lines were plotted, in the right plot $60$). For non-smooth initial conditions, a strong regularization is required. Here, no convergence is observed for $\varepsilon \le 10^{-3}$.}
    \label{fig:defects}
\end{figure}

\subsection*{Long-time behaviour of the numerical solution}
Consider now the setting of the transport equation \eqref{eq:transport}, with periodic boundary conditions and the regular initial value $u_0$. We note that the exact solution simply transports the initial values and is, due to the periodic boundary conditions imposed at the boundary points of $\Omega=[-\pi,\pi]$, periodic with the period $2\pi$. For initial values that are determined by the parametrization, the exact solution can be exactly resolved by the parametrization. 

Figures~\ref{fig:longtime-1} and~\ref{fig:longtime-2} show the long-time behaviour of the Gau\ss \ implicit Runge--Kutta method with $s=2$ stages. We use the fixed time stepsize $h = \tfrac{2\pi}{100}$ (which corresponds to $100$ steps per full period). In order to roughly fulfill the step size restriction, we set $\varepsilon=10^{-2}$. The approximation quality remains high for a long time, which is in contrast to physics-informed neural networks (PINNs), where specific modifications need to be derived for the effective construction of good approximations over longer time periods \cite{RBMB23}.


\begin{figure}
	\centering
	\includegraphics[width=1\textwidth,trim = 0mm 0mm 0mm 0mm]{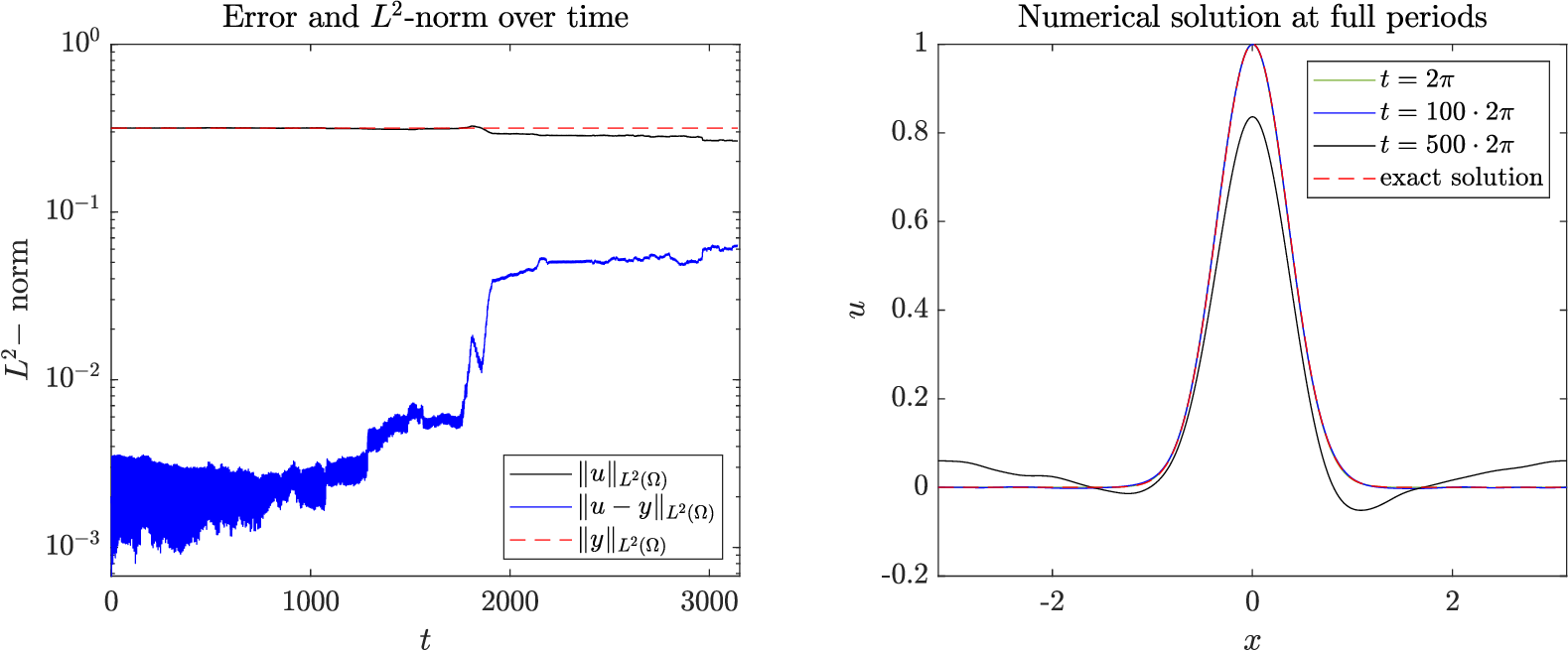}
	\caption{Approximating the transport equation with the $2$-stage Gauß method over long time. The error remains small for a remarkably long time. The computation used $N=200$ time steps per period, such that $h = \frac{2\pi}{200}$.
    }
	\label{fig:longtime-1}
\end{figure}

\begin{figure}[h!]
\centering	\includegraphics[width=1.05\textwidth,trim = 20mm 0mm 0mm 0mm]{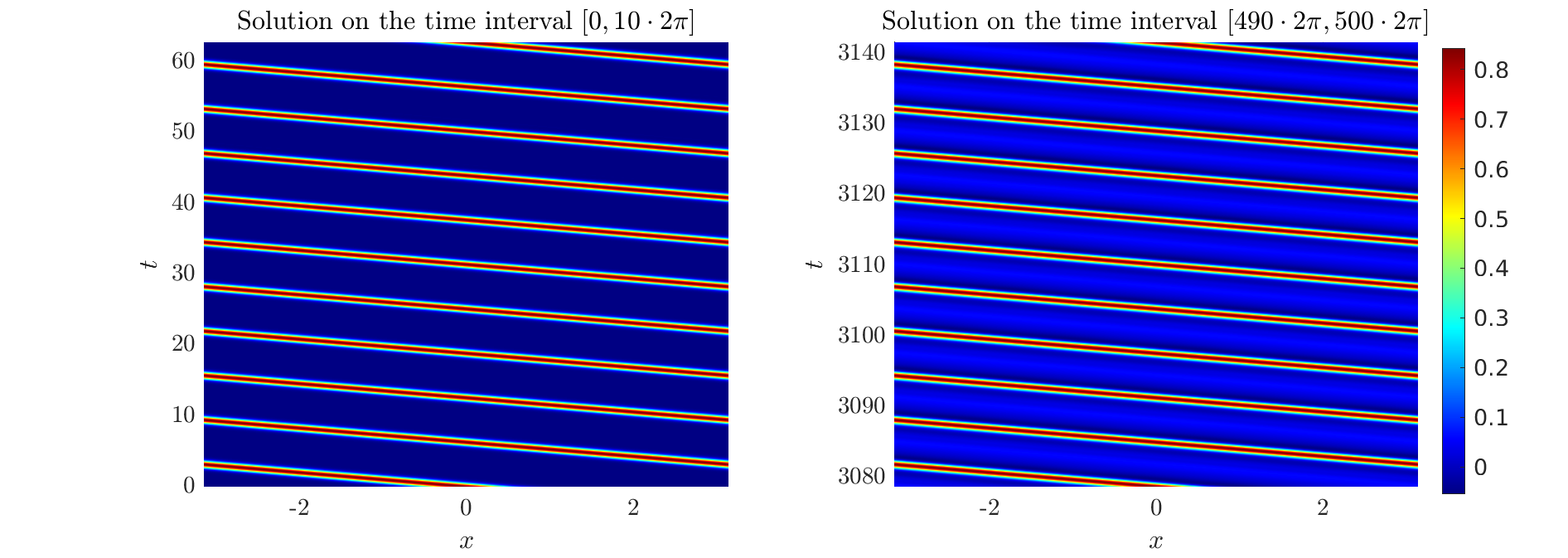}
	\caption{The numerical solution on different time intervals, on the left the first $10$ periods and on the right the periods $490-500$. The numerical solution exhibits similar behaviour as the exact solution, for a relatively long timespan.}
	\label{fig:longtime-2}
\end{figure}


	\begin{acknowledgement} We thank Ernst Hairer for fruitful discussions, which contributed substantially to the inclusion of Runge--Kutta methods in the manuscript. This work was funded by the Deutsche Forschungsgemeinschaft (DFG, German Research Foundation) under project TRR 352,  Project-ID 470903074. 
	\end{acknowledgement}

\bibliographystyle{abbrv}
\bibliography{Lit.bib}

\end{document}